\documentclass[a4paper,11pt]{article}
\usepackage{geometry}

\usepackage[english]{babel}
\usepackage[utf8]{inputenc}

\usepackage{libertinus} 

\usepackage{authblk}
\usepackage{mathrsfs}
\usepackage{amsfonts, amssymb, amsmath, amsthm}
\usepackage{amsbsy}

\usepackage[backref=page]{hyperref}

\usepackage{tikz}
\usepackage{tabularx}

\usepackage{fancyhdr}
\usepackage{ifthen}
\newtheoremstyle{mystyle}{}{}{\rmfamily}%
{}{\normalfont\bfseries}{ }{ }{}

\hypersetup{colorlinks=true, linkcolor=blue, citecolor=blue, urlcolor=blue}

\theoremstyle{mystyle}
\newtheorem{theorem}{Theorem}[section]
\newtheorem*{theorem*}{Theorem}
\newtheorem{lemma}[theorem]{Lemma}
\newtheorem{definition}[theorem]{Definition}

\newtheorem{corollary}[theorem]{Corollary}
\newtheorem{proposition}[theorem]{Proposition}
\newtheorem{remark}[theorem]{Remark}

\newtheorem{question}[theorem]{Question}
\newtheorem{example}[theorem]{\bf Example}
\newtheorem{problem}[theorem]{Problem}

\newcommand{\keywords}[1]{\par\textbf{Keywords:} #1\par}

\title{
On Spaceability within Linear Dynamics
}
\author[1]{Manuel Saavedra}
\author[2]{Manuel Stadlbauer}
\affil[1,2]{Instituto de Matemática, Universidade Federal do Rio de Janeiro, RJ, Brazil.}
\affil[1]{\texttt{saavmath@im.ufrj.br}}
\affil[2]{\texttt{manuel@im.ufrj.br}}

\date{}

\fancypagestyle{customstyle}{
	\fancyhf{}
	\fancyhead[LE,RO]{\thepage}
	\fancyhead[CE,CO]{\ifthenelse{\isodd{\thepage}}{{Manuel Saavedra and Manuel Stadlbauer}}{On Spaceability within Linear Dynamics}}
	
}

\begin{document}
	\maketitle
\setlength{\headheight}{14.49998pt}

\begin{center}
	\textit{The first author gratefully dedicates this work to his father.}
\end{center}

\begin{abstract}	
We investigate spaceability phenomena in linear dynamics from a structural perspective. 
Given a continuous linear operator \(T:X \to X\), we introduce the set \(\Omega(T)\), consisting of all continuous linear operators \(h:X \to X\) for which there exists a strictly increasing sequence \((\theta_n)_n\) of positive integers such that the set  \(\{x \in X : \displaystyle{\lim_{n \to \infty} T^{\theta_n}x = h(x)}\}\) is dense in \(X\). Within this framework, two classical phenomena—the existence of hypercyclic and recurrent subspaces in separable infinite-dimensional complex Banach spaces—emerge as instances of a common underlying structure described by \(\Omega(T)\). To analyze \(\Omega(T)\), we introduce the notion of collections simultaneously approximated (c.s.a.) by \(T\), and show that every maximal c.s.a. is an SOT-closed affine manifold. For quasi-rigid operators on separable Banach spaces, we establish the existence of a unique maximal c.s.a. containing the identity operator. Furthermore, we examine \(\Omega(T)\) through the left-multiplication operator \(L_T\) acting on the algebra of bounded operators. Our approach combines two key ingredients: a refinement of A. López’s technique on recurrent subspaces for quasi-rigid operators, and a common dense-lineability result obtained by the first author and A. Arbieto. These tools yield new spaceability results for the sets \(\Omega(T)\), \(\mathcal{AP}\Omega(T)\), and for any countable c.s.a. by \(T\).
\end{abstract}

	
\section{Introduction}

The study of spaceability within linear dynamics has been a central theme over the last decades. Among the most investigated phenomena are the existence of hypercyclic and recurrent subspaces, which have attracted considerable attention. In this direction, A. López \cite{Lopez} observed that, at least for the spaceable property, hypercyclicity and recurrence can be treated as equals. This naturally raises the question: can hypercyclic and recurrent subspaces be regarded as particular instances of a broader structural scheme?

In this paper, we provide an affirmative answer by showing that both notions emerge as concrete realizations of a structural setting naturally captured by \(\Omega(T)\), which will be further examined in the sequel.

To provide context, we recall some basic notions. Let $X$ be an infinite-dimensional $F$-space. A subset $A \subset X$ is said to be dense-lineable if $A \cup \{0\}$ contains a dense linear subspace, and spaceable if $A \cup \{0\}$ contains an infinite-dimensional closed subspace of $X$, following the terminology used in \cite{aron2005lineability, bernal2014lineability, bernal-linear, gurariy2004lineability, seoane2006chaos}.

We say that an infinite-dimensional closed subspace $Z \subset X$ is a hypercyclic subspace for $T$ if $Z \subset \mathrm{HC}(T) \cup \{0\}$, where $\mathrm{HC}(T)$ denotes the set of all hypercyclic vectors for $T$. Similarly, an infinite-dimensional closed subspace $Z \subset X$ is a recurrent subspace for $T$ if $Z \subset \mathrm{Rec}(T) \cup \{0\}$, where $\mathrm{Rec}(T)$ denotes the set of all recurrent vectors for $T$. In other words, a hypercyclic operator $T$ admits a hypercyclic subspace if and only if $\mathrm{HC}(T)$ is spaceable, and a recurrent operator $T$ admits a recurrent subspace if and only if $\mathrm{Rec}(T)$ is spaceable.

The first work to establish sufficient conditions for the existence of a hypercyclic subspace was due to A. Montes-Rodríguez \cite{Montes}. Subsequently, equivalences between the existence of hypercyclic subspaces and certain properties of the essential spectrum for operators satisfying the hypercyclicity criterion were investigated. In Hilbert spaces, these equivalences were demonstrated by F. León-Saavedra and A. Montes-Rodríguez \cite{Leon}, and later generalized to Banach spaces by M. González, F. León-Saavedra, and A. Montes-Rodríguez \cite{Gonzales}.

More recently, A. López \cite{Lopez} investigated recurrent subspaces for quasi-rigid operators on Banach spaces, highlighting structural properties analogous to those observed for hypercyclic subspaces. 

Recall that a continuous linear operator $T:X \to X$ is quasi-rigid if there exists a strictly increasing sequence of positive integers $(\theta_n)_n$ such that the set \(\{x \in X : T^{\theta_n} x \xrightarrow[n\rightarrow \infty]{} x\}\) is dense in $X$.

Our change of perspective relies on the set \(\Omega(T) \subset \mathcal{L}(X)\), defined as the collection of all continuous linear operators $h: X \to X$ for which there exists a strictly increasing sequence of positive integers $(\omega_n)_n$ such that the set
\[
\{x \in X : T^{\omega_n} x \xrightarrow[n\rightarrow \infty]{} h(x)\}
\]
is dense in $X$.

Notably, when $X$ is a separable infinite-dimensional Fréchet or Banach space, the set $\Omega(T)$ provides a simple characterization of important dynamical properties:
\begin{align*}
	T \text{ is quasi-rigid} & \quad \Longleftrightarrow \quad \mathrm{Id} \in \Omega(T), \\
	T \text{ is weakly mixing} & \quad \Longleftrightarrow \quad \Omega(T) = \mathcal{L}(X).
\end{align*}

To establish the connection with spaceability, for a continuous linear map $h:X \to X$ we define
\[
\text{R}(T,h) := \{x \in X : \exists\, (\omega_n)_n \uparrow \infty \text{ such that } \lim_{n \to \infty} T^{\omega_n} x = h(x)\}.
\]

The following two cases illustrate how the theorem below unifies and recovers the classical notions of hypercyclic and recurrent subspaces.

\begin{table}[h]
	\centering
	\renewcommand{\arraystretch}{1.3}
	\begin{tabular}{p{0.44\textwidth}|p{0.44\textwidth}}
		\hspace{2cm} \text{$T$ weakly mixing} & \hspace{2cm} \text{$T$ quasi-rigid} \\ \hline
		\hspace{1cm} 	$\Omega(T)=\mathcal{L}(X)$, SOT-separable & \hspace{1cm} $F=\{\mathrm{Id}\}\subset \Omega(T)$, SOT-separable \\[4pt]
		\hspace{1cm}	$\displaystyle \bigcap_{h\in \mathcal{L}(X)} \text{R}(T,h) \;=\; \mathrm{HC}(T)\cup \{0\}$ 
		& 
		\hspace{1cm} $\displaystyle \bigcap_{h\in\{\mathrm{Id}\}} \text{R}(T,h) \;=\; \mathrm{Rec}(T)$
	\end{tabular}
\end{table}

\begin{theorem*}
	Let $X$ be a complex separable infinite-dimensional Banach space, and let $T \in \mathcal{L}(X)$. Suppose that $\Omega(T)$ is non-empty. If there exists a strictly increasing sequence of positive integers $(\theta_n)_n$ and an infinite-dimensional closed subspace $E \subset X$ such that
	\[
	\sup_n \|T^{\theta_n}|_E\| < \infty,
	\]
	then for any SOT-separable subset $F \subset \Omega(T)$,
	\[
	\bigcap_{h \in F} \text{R}(T, h)
	\]
	is spaceable.
\end{theorem*}

Our approach relies on a refinement of the technique developed by López \cite{Lopez}, combined with a common dense-lineability result recently obtained by the first author together with Arbieto \cite{arbieto2025dense}. In the same spirit, we obtain an analogue of the preceding theorem by considering the set $\mathcal{AP}\Omega(T)$ associated with the Furstenberg family $\mathcal{AP}$.

In Section \ref{section2}, we introduce the set $\Sigma(T)$, which naturally arises in the study of recurrence phenomena. We derive sufficient conditions ensuring that $\Sigma(T)$ encodes the recurrence of $T$ (Corollary~\ref{T-Sigma}). Furthermore, we address the interplay of $\Sigma(T)$ with two recently investigated questions---the $T \times T$-recurrence problem \cite{Sophie} and the non dense-lineability of $\text{Rec}(T)$ \cite{LopezM}. We show that every separable infinite-dimensional complex Banach space supports a recurrent operator $T$ such that $\Sigma(T \oplus T)=\emptyset$, and moreover, for each $h \in \Sigma(T)$ the set $\text{R}(T,h)$ fails to be dense-lineable (Theorem~\ref{Sigma-Tapia}).

In Section \ref{section3}, we investigate structural aspects of the set $\Omega(T)$. We introduce the notion of a collection simultaneously approximated by $T$, see Definition \ref{def-simul}. We show that every maximal collection simultaneously approximated by $T$ is an SOT-closed affine manifold in $\mathcal{L}(X)$, see Theorem \ref{convex}. When $T$ is quasi-rigid on a separable Banach space, there exists a unique maximal collection simultaneously approximated by $T$ containing the identity operator, whose intersection with $\mathrm{GL}(X)$ is locally convex and forms a normal subgroup of $\Omega(T)\cap \mathrm{GL}(X)$, see Theorem \ref{N-G}.  

A further relevant result involving the left-multiplication operator $L_{T}$ in the algebra $\mathcal{L}(X)$ establishes that, on a separable infinite-dimensional Banach space, $L_T$ is $\mathrm{SOT}$-hypercyclic if and only if $T$ satisfies the Hypercyclicity Criterion \cite{livro, chan1999hypercyclicity, Chan, Grosse}. This motivates us to study the nature of $\Omega(T)$ from the perspective of the operator $L_T$.  

\begin{theorem*}
	Let $X$ be a separable Banach space, and let $T$ be a quasi-rigid operator on $X$. Then
	\[
	\Omega(T) \subset \overline{\bigcup_{n \in \mathbb{N}} L_T^n(A)}^{\mathrm{SOT}}
	\]
	for every open set $A \subset (\mathcal{L}(X), \mathrm{SOT})$ that contains a surjective operator in $\Omega(T)$.
\end{theorem*}

In Section \ref{section4}, we establish our main spaceability result for the sets $\Omega(T)$, $\mathcal{AP}\Omega(T)$, and for any countable collection simultaneously approximated by $T$.

\section{The Framework of \(\Sigma(T)\)}\label{section2}

This section has two purposes. The first is to set the stage for the study of $\Omega(T)$ in the next section. The second is to show how the set $\Sigma(T)$ provides a structural perspective on classical notions such as recurrence and hypercyclicity.

Let \(X\) be a complete metric space and let \(T:X \to X\) be a continuous map. When \(X\) is an infinite-dimensional \(F\)-space (i.e. \(X\) is a completely metrizable topological vector
space), we shall restrict our attention to the case where \(T\) is a continuous linear operator on \(X\). Throughout the paper, we denote by \(\mathcal{L}(X)\) the space of all continuous linear operators acting on \(X\).

Given a continuous map \(g:X \to X\), we define
\begin{align*}
	\text{R}(T,g) := \{x \in X : \exists\, \omega_n \uparrow \infty \text{ such that } \lim_{n \to \infty} T^{\omega_n}x = g(x)\}.
\end{align*}
We then introduce the set
\[
\Sigma(T) := \{ g:X \to X \ \text{continuous}\, \text{such that}\; \text{R}(T,g) \text{ is dense in } X\}.
\]
When \(T\) is a continuous linear operator on an \(F\)-space \(X\), we shall, by a slight abuse of notation, denote by \(\Sigma(T)\) the subset of \(\mathcal{L}(X)\) consisting of all \(h \in \mathcal{L}(X)\) such that \(\text{R}(T,h)\) is dense in \(X\).

Recall that a continuous linear operator \(T : X \to X\) on a separable Banach space is called \emph{hypercyclic} if there exists a vector \(x \in X\) such that the orbit \(\{T^n x : n \in \mathbb{N}\}\) is dense in \(X\). The set of such vectors is denoted by \(\mathrm{HC}(T)\); see the books \cite{livro, Grosse} for detailed accounts. 

A continuous map \(T\) on \(X\) is said to be \emph{recurrent} if the set of recurrent points, denoted by \(\mathrm{Rec}(T)\), is dense in \(X\). Here, a point \(x \in X\) is recurrent provided there exists a strictly increasing sequence \((\theta_n)_{n \in \mathbb{N}}\) such that \(\displaystyle{\lim_{n \to \infty} T^{\theta_n}x = x}\). For general aspects of recurrence in topological dynamics, we refer to \cite{Furs, gottschalk1955topological}. In the linear setting, recurrent operators were introduced and systematically investigated by Costakis, Manoussos, and Parissis \cite{Cos, Cos2}. It follows immediately from the definitions that \(T\) is recurrent if and only if \(\mathrm{Id} \in \Sigma(T)\).

Consequently, when \(X\) is a separable infinite-dimensional Fréchet space, we obtain
\begin{align*}
	T \text{ is recurrent} & \hspace{0.9cm} \Longleftrightarrow \hspace{0.9cm} \mathrm{Id} \in \Sigma(T), \\
	T \text{ is hypercyclic} & \quad \underset{\text{Proposition ~\ref{hyper-sigma}}}{\Longleftrightarrow} \quad \mathcal{L}(X)=\Sigma(T).
\end{align*}

\begin{proposition}\label{open-sigma}
	Let \(X\) be a complete metric space, and let \(T \colon X \to X\) be continuous. For a continuous map \(g \colon X \to X\), the following are equivalent:
	\begin{enumerate}
		\item[(i)] \( g \in \Sigma(T) \).
		\item[(ii)] For every pair of open sets \(U, V \subset X\) with \( g(U) \cap V \neq \emptyset \), the set
		\[
		\{ m \in \mathbb{N} : U \cap T^{-m}(V) \neq \emptyset \}
		\]
		is infinite.
	\end{enumerate}
\end{proposition}

\begin{proof}
	Assume (i). Let \(U, V \subset X\) be non-empty open sets with \(g(U) \cap V \neq \emptyset\). By continuity of \(g\), there exists an open set \(W \subset U\) with \(g(W) \subset V\). Since \(\text{R}(T,g)\) is dense, we may choose \(x \in W \cap \text{R}(T,g)\). Thus, there is a strictly increasing sequence \((\theta_n)\) such that \(T^{\theta_n}(x) \to g(x) \in V\). Hence,
	\[
	\{\theta_n : n \geq n_0\} \subset \{ m \in \mathbb{N} : U \cap T^{-m}(V) \neq \emptyset \}
	\]
	for some \(n_0\), proving (ii).
	
	Conversely, assume (ii). Fix \(x \in X\) and \(\epsilon > 0\). Since \(g(B(x,\epsilon)) \cap B(g(x),\epsilon) \neq \emptyset\), condition (ii) yields \(m \in \mathbb{N}\) with \(B(x,\epsilon) \cap T^{-m}(B(g(x),\epsilon)) \neq \emptyset\). Following the inductive construction in \cite[Proposition~2.1]{Cos} with \(x_{0}:=x\) and \(\epsilon_{0}:=\epsilon\), one can build sequences \( (x_\ell) \subset X\), \( \epsilon_\ell \downarrow 0\), and \( \theta_\ell \uparrow \infty\) such that for each \(\ell\in \mathbb{N}\):
	\begin{align} \label{induc-1}
			\overline{B(x_\ell, \epsilon_\ell)} \subset B(x_{\ell-1}, \epsilon_{\ell-1}) 
		\quad \text{and} \quad 
		T^{\theta_\ell}(B(x_\ell, \epsilon_\ell)) \subset B(g(x_{\ell-1}), \epsilon_{\ell-1}).
	\end{align}
	By the Cantor intersection theorem, there exists \(\{y\}:=\bigcap_\ell B(x_\ell, \epsilon_\ell)\). Then \(	T^{\theta_\ell}(y) \in B(g(x_{\ell-1}), \epsilon_{\ell-1})\) for each \(\ell\in \mathbb{N}\), which implies that \( T^{\theta_\ell}(y) \to g(y) \) as \( \ell \to \infty \). Hence \( y \in \text{R}(T,g) \), and consequently \( \text{R}(T,g) \) is dense in \( X \). Therefore, (i) follows.
\end{proof}

A stronger notion than recurrence is that of \emph{multiple recurrence}, introduced by Furstenberg in the framework of topological dynamics. This concept has deep connections with ergodic theory, number theory, and combinatorics. Recall that a continuous map \(T\) on a complete metric space \(X\) is said to be (topologically) multiply recurrent if, for every non-empty open set \(U \subset X\) and every \(m\in \mathbb{N}\), there exists \(r\in \mathbb{N}\) such that 
\begin{align*}
	U\cap T^{-r}U\cap \cdots\cap T^{-mr}U\neq \emptyset.
\end{align*} 
In the context of linear dynamics, the first systematic study of multiple recurrence was carried out by Costakis and Parissis \cite{Cos2}. 

More recently, it was established that multiple recurrence coincides with the notion of \emph{\(\mathcal{AP}\)-recurrence} \cite{cardeccia2022multiple, kwietniak2017multi}. To place this in context, recall the Furstenberg family \(\mathcal{AP} \subset \mathcal{P}(\mathbb{N})\), defined by
\begin{align*}
	\mathcal{AP} := \{\,\text{B} \subset \mathbb{N}\cup \{0\} :  \text{B} \text{ contains arbitrarily long finite arithmetic progressions}\,\}.
\end{align*}
For \(x\in X\) and a non-empty open set \(U\subset X\), we set
\[
\mathcal{N}_{T}(x,U):=\{n\in \mathbb{N}\cup\{0\} : T^{n}x\in U\}.
\]
A point \(x\in X\) is called \(\mathcal{AP}\)-recurrent for \(T\) if, for every neighbourhood \(U\) of \(x\), one has \(\mathcal{N}_{T}(x, U)\in \mathcal{AP}\). We denote by \(\mathcal{AP}\text{Rec}(T)\) the set of all \(\mathcal{AP}\)-recurrent points of \(T\), and we say that \(T\) is \(\mathcal{AP}\)-recurrent if \(\mathcal{AP}\text{Rec}(T)\) is dense in \(X\). Moreover, it is shown in \cite[Lemma 4.8]{kwietniak2017multi} that if \(T\) is \(\mathcal{AP}\)-recurrent, then \(\mathcal{AP}\text{Rec}(T)\) is residual in \(X\).

Recall that \(T\) is said to be \(\mathcal{AP}\)-hypercyclic if there exists \(x \in X\) such that, for every non-empty open set \(U \subset X\), we have \(\mathcal{N}_{T}(x,U) \in \mathcal{AP}\).  
The following result, obtained by R. Cardeccia and S. Muro \cite{cardeccia2022multiple}, reveals connections between \(\mathcal{AP}\)-hypercyclicity and multiple recurrence in the hypercyclic setting:
\begin{proposition}[\cite{cardeccia2022multiple}] \label{AP-Hyper}
	Let \(T\) be a linear operator on a separable Fréchet space. Then the following statements are equivalent:
	\begin{enumerate}
		\item \(T\) is hypercyclic and every hypercyclic vector is \(\mathcal{AP}\)-hypercyclic.
		\item There exists an \(\mathcal{AP}\)-hypercyclic vector.
		\item \(T\) is hypercyclic and multiply recurrent.
		\item For every pair of nonempty open sets \(U,V \subset X\) and each \(m \in \mathbb{N}\), there exist \(a,r \in \mathbb{N}\) such that
		\begin{align*}
			U \cap \left(\bigcap_{\ell=1}^{m} T^{-(a+\ell r)} V\right) \neq \emptyset.
		\end{align*}
		\item The set of \(\mathcal{AP}\)-hypercyclic vectors is residual in \(X\).
	\end{enumerate}
\end{proposition}

Part of the equivalences in the previous result are encompassed within a broader framework, namely the Birkhoff Theorem for upper Furstenberg families established by A.~Bonilla and K.-G.~Grosse-Erdmann \cite{bonilla2018upper}.

For a continuous map $h$ on $X$, we define
\begin{align*}
	\mathcal{AP}\text{R}(T,h) 
	&:= \{x \in X : \forall\, V \ni h(x) \text{ open}, \, N_{T}(x,V) \in \mathcal{AP}\}\\
	&= \bigcap_{k\in \mathbb{N}} \bigcup_{a,n\in \mathbb{N}} 
	\left\{x \in X : d(T^{a+\ell n}x, h(x)) < \tfrac{1}{k},\ \ell \in \{1,\ldots, k\}\right\}\\
	&=\{x\in X: \exists\, \text{B}\in \mathcal{AP}\, \text{such that}\, \lim_{\substack{n\rightarrow \infty \\ n\in \text{B} }}T^{n}x=h(x)\}
\end{align*}
The above characterization immediately implies that the set \(\mathcal{AP}\text{R}(T,h)\) is a \(G_{\delta}\)-set.

\begin{definition}
	Let $X$ be a complete metric space and $T$ a continuous map on $X$. We set
	\begin{align*}
		\mathcal{AP}\Sigma(T) := \{\, h : X \to X \text{ continuous}\, \text{such that}\, \mathcal{AP}\text{R}(T,h) \text{ is dense in } X \,\}.
	\end{align*}
\end{definition}

\begin{proposition}\label{AP}
	Let $X$ be a complete metric space and let $T$ be a continuous map on $X$. 
	For a continuous map $h \colon X \to X$, the following are equivalent:
	\begin{enumerate}
		\item $h \in \mathcal{AP}\Sigma(T)$.
		\item For any non-empty open sets $U,V \subset X$ with $h(U) \cap V \neq \emptyset$ and for each $m \in \mathbb{N}$, there exist $a,r \in \mathbb{N}$ such that
		\[
		U \cap \left(\bigcap_{\ell=1}^{m} T^{-(a+\ell r)}V\right) \neq \emptyset.
		\]
	\end{enumerate}
\end{proposition}

\begin{proof}
	The implications (1) $\Rightarrow$ (2) and (2) $\Rightarrow$ (1) follow the same general strategy as in the proof of Proposition~\ref{open-sigma}. 
	
	For (2) $\Rightarrow$ (1), fix $x_{0}\in X$ and $\epsilon_{0}>0$. Proceeding inductively, we construct a sequence $\{x_n\}\subset X$, a decreasing sequence $\{\delta_n\}$ with $\epsilon_n \downarrow 0$, and two sequences of positive integers $\{a_n\}$ and $\{r_n\}$ such that:
	\begin{itemize}
		\item $\overline{B(x_{n},\epsilon_{n})} \subset B(x_{n-1},\delta_{n-1})$ for all $n\in \mathbb{N}$,
		\item $T^{a_{n}+\ell r_{n}}(B(x_{n},\epsilon_{n})) \subset B(h(x_{n-1}),\epsilon_{n-1})$ for $\ell=0,1, \dots,n$ and each \(n\in \mathbb{N}\)
		\item $a_{n+1} > a_{n} + n r_{n}$ for all $n\in \mathbb{N}$.
	\end{itemize}
	
	Let $\{\theta_n\}$ be the increasing enumeration of the set $\{a_m+\ell r_m : m \in \mathbb{N}, \ \ell=0,1,\dots,m\}$. By the Cantor intersection theorem, \(	\{q\}:= \bigcap_{m} \overline{B(x_m,\epsilon_m)} \subset B(x_{0},\epsilon_{0})\). Hence $q \in \mathcal{AP}\text{R}(T,h)\cap B(x_{0}, \epsilon_{0})$, and therefore $h \in \mathcal{AP}\Sigma(T)$.
\end{proof}

\begin{proposition}\label{hyper-sigma}
	Let \(T\) be a continuous linear operator on a separable infinite-dimensional Fréchet space \(X\). Then \(T\) is hypercyclic (resp. \(\mathcal{AP}\)-hypercyclic) if and only if \(\Sigma(T) = \mathcal{L}(X)\) (resp. \(\mathcal{AP}\Sigma(T) = \mathcal{L}(X)\)).
\end{proposition}

\begin{proof}
	We prove the statement for the \(\mathcal{AP}\)-case, as the hypercyclic case follows by the same argument.  
	If \(T\) is \(\mathcal{AP}\)-hypercyclic, then \(\mathcal{AP}\text{HC}(T)\subset \mathcal{AP}\text{R}(T,h)\) for each \(h\in \mathcal{L}(X)\). 
	
	For the converse, assume \(\mathcal{AP}\Sigma(T) = \mathcal{L}(X)\).  
	Fix two nonempty open sets \(U,V \subset X\). Since \(X\) is a Fréchet space, there exists a continuous linear operator \(h \colon X \to X\) such that \(h(U)\cap V \neq \emptyset\).  
	By Proposition~\ref{AP}, for each \(m\in \mathbb{N}\) there exist \(a,r \in \mathbb{N}\) such that
	\[
	U \cap \left( \bigcap_{\ell=1}^{m} T^{-(a+\ell r)} V \right) \neq \emptyset.
	\]
	Applying Proposition~\ref{AP-Hyper}, we conclude that \(T\) is \(\mathcal{AP}\)-hypercyclic.
\end{proof}

\begin{proposition}\label{Sigma-SOT}
	Let \(X\) be an \(F\)-space and \(T \in \mathcal{L}(X)\). Then both \(\Sigma(T)\) and \(\mathcal{AP}\Sigma(T)\) are closed in \(\mathcal{L}(X)\) with respect to the strong operator topology.
\end{proposition}

\begin{proof}
	Assume that \(X\) is an \(F\)-space. We first show that \(\Sigma(T)\) is SOT-closed. Fix \(g \in \overline{\Sigma(T)}^{\mathrm{SOT}}\). Fix non-empty open sets \(U, V \subset X\) with \(g(U) \cap V \neq \emptyset\). Choose \(p \in U\) and \(\epsilon > 0\) such that \(B(g(p), \epsilon) \subset V\). Consider the SOT-neighborhood of \(g\) given by
	\[
	N(g, p, \epsilon) := \{f \in \mathcal{L}(X) : d(g(p), f(p)) < \epsilon\}.
	\]
	Since \(g \in \overline{\Sigma(T)}^{\mathrm{SOT}}\), there exists \(h \in N(g, p, \epsilon) \cap \Sigma(T)\). Note that \(h(p) \in h(U) \cap V\). Moreover, because \(h \in \Sigma(T)\), the set \(\{m \in \mathbb{N} : U \cap T^{-m}(V) \neq \emptyset\}\) is infinite. Hence, by Proposition \ref{open-sigma}, we deduce that \(g \in \Sigma(T)\). This proves that \(\Sigma(T)\) is SOT-closed. A similar argument shows that \(\mathcal{AP}\Sigma(T)\) is also SOT-closed in \(\mathcal{L}(X)\).
\end{proof}



\subsection{A Sufficient Conditions on \(\Sigma(T)\) Implying Recurrence of \(T\)}

A natural and motivating question arises: can we decide whether $T$ is recurrent whenever $ \Sigma(T) $ contains a specific operator, such as $ \tfrac{1}{2}\mathrm{Id} $? Formally:

\begin{problem}\label{P1}
	Let $X$ be an $F$-space and $T \in \mathcal{L}(X)$. If $ \tfrac{1}{2}\mathrm{Id} \in \Sigma(T) $, does it follow that $T$ is recurrent?
\end{problem}

The particular choice of $\tfrac{1}{2}$ is inessential; any nonzero complex scalar could be used instead, leading to the same affirmative conclusion.

\begin{theorem}\label{T-Sigma}
	Let $X$ be a complete metric space and $T$ a continuous map on $X$. Then, \(T\) is recurrent if and only if \(\Sigma(T)\) contains a continuous map with dense range.
\end{theorem}

\begin{proof}
	Assume that $\Sigma(T)$ contains a continuous map $h$ with dense range. Fix any nonempty open set $V \subset X$. Since $h$ has dense range, there exists a nonempty open $U \subset X$ such that $h(U) \subset V$. By Proposition \ref{open-sigma}, there is $m \in \mathbb{N}$ with $U \cap T^{-m}V \neq \emptyset$. Set $W := U \cap T^{-m}V$, so that $h(W) \subset V$. Again, Proposition \ref{open-sigma} ensures that \(\{\ell > m : W \cap T^{-\ell}V \neq \emptyset\}\) 	is infinite. Hence, \(	\{k \in \mathbb{N} : V \cap T^{-k}V \neq \emptyset\}\) 	is also infinite. Since $V$ was arbitrary, $T$ is recurrent.
\end{proof}

\begin{corollary}
	Let $X$ be a complete metric space and $T$ a continuous map on $X$. Then, \(T\) is \(\mathcal{AP}\)-recurrent if and only if \(\mathcal{AP}\Sigma(T)\) contains a continuous map with dense range.
\end{corollary}
 
Let us assume that $\Sigma(T)$ contains some operator with dense range $h:X\to X$. By Theorem \ref{T-Sigma}, this already guarantees that $T$ is recurrent. This naturally raises the question of whether $h$ may also serve as a bridge to recover further information about the set $\Sigma(T)$. More precisely, given $g \in \Sigma(T)$, one may ask whether the set
\[
\Big\{ x \in X : \exists\; \theta_n \uparrow \infty \text{ such that } T^{\theta_n} h(x) \xrightarrow[n \to \infty]{} g(x) \Big\}
\]
is dense in $X$. Unfortunately, having dense range is not sufficient, as the following example shows.

\begin{example}
	Consider the weighted backward shift operator \( B \) on \( X := \ell^2(\mathbb{N}) \), defined by \( B(e_1) = 0 \) and \( B(e_n) = 2e_{n-1} \) for \( n \geq 2 \), where \((e_n)_{n \in \mathbb{N}}\) denotes the canonical basis of \(\ell^2(\mathbb{N})\). It is known (see \cite[Theorem 1.40]{livro}) that \( B \) is hypercyclic. In this case, we have \(\Sigma(B) = \mathcal{L}(X)\) by Proposition 
	
	Now consider the continuous linear operator \( A \) on \( X \) with dense range, defined by
	\begin{align*}
		A: \ell^2(\mathbb{N}) & \longrightarrow \ell^2(\mathbb{N}), \\
		(x_i)_{i \in \mathbb{N}} & \longmapsto \left( \frac{x_i}{3^i} \right)_{i \in \mathbb{N}}.
	\end{align*}
	Fix any \( x \in X \). Let us examine the \( B \)-orbit of the vector \( Ax \). For \( n \in \mathbb{N} \),
	\[
	B^n \left( \frac{x_i}{3^i} \right)_{i} 
	= 2^n \left( \frac{x_{i+n-1}}{3^{i+n-1}} \right)_{i}
	\xrightarrow[n \to \infty]{} 0.
	\]
\end{example}

\begin{proposition}\label{relations}
	Let \(X\) be a complete metric space and \(T\) a continuous map on \(X\). If \(\Sigma(T)\) contains a continuous open map \(h \colon X \to X\), then for every \(g \in \Sigma(T)\), the set
	\begin{align}\label{h-g}
		\{\, x \in X : \exists\, \theta_n \uparrow \infty \text{ with } T^{\theta_n} h(x) \to g(x) \,\}
	\end{align}
	is residual in \(X\).
\end{proposition}

\begin{proof}
	Fix \(h \in \Sigma(T)\) continuous and open, and \(g \in \Sigma(T)\).  
	The set in \eqref{h-g} is clearly a \(G_\delta\); it remains to show density.  
	
	Let \(U \subset X\) be a non-empty open set. Choose \(x_1 \in \text{R}(T,h) \cap \text{R}(T,g)\) and \(\epsilon > 0\) such that \(B(x_1,\epsilon) \subset U\). Since \(h\) is open, \(h(B(x_1,\epsilon))\) is an open set containing \(h(x_1)\). Because \(x_1 \in \text{R}(T,h)\), there exists \(\psi_1 \in \mathbb{N}\) with \(	T^{\psi_1} x_1 \in h(B(x_1,\epsilon))\). By continuity of \(T\), one can choose \(0 < \delta_1 < \min\{2^{-1},\epsilon\}\) such that \(T^{\psi_1}(B(x_1,\delta_1)) \subset h(B(x_1,\epsilon))\).

	Now consider the open sets \(B(g(x_1),2^{-1})\) and \(B(x_1,\delta_1)\).  
	By Proposition~\ref{open-sigma} applied to \(g\), there exists \(\omega_1 > \psi_1\) with
	\begin{align*}
	\emptyset \neq B(g(x_1),2^{-1}) \cap T^{\omega_1}(B(x_1,\delta_1)) 
	\subset B(g(x_1),2^{-1}) \cap T^{\omega_1-\psi_1} h(B(x_1,\epsilon)).
	\end{align*}
	Thus we may select \(x_2 \in \text{R}(T,h) \cap \mathrm{R}(T,g)\) and \(0<\delta_2<2^{-2}\) such that
	\begin{align*}
	\overline{B(x_2,\delta_2)} \subset B(x_1,\delta_1), 
	\qquad 
	T^{\omega_1-\psi_1} h(B(x_2,\delta_2)) \subset B(g(x_1),2^{-1}).
	\end{align*}

	Iterating this construction, we obtain sequences \(\{x_j\}_{j \in \mathbb{N}} \subset X,
	\delta_j \downarrow 0,  
	\theta_j \uparrow \infty,\) such that
	\begin{itemize}
    \item \(\overline{B(x_{j+1},\delta_{j+1})} \subset B(x_j,\delta_j)\) for each \(j\in \mathbb{N}\),
    \item \(T^{\theta_{j+1}} h(B(x_{j+1},\delta_{j+1})) \subset B(g(x_j),2^{-j}), 
	\) for each \(j\in \mathbb{N}\).
	\end{itemize}
	By Cantor’s intersection theorem, \(\{y\} := \bigcap_{j} B(x_j,\delta_j)\) with \(y \in U\). Moreover,
	\begin{align*}
		d(T^{\theta_{j+1}} h(y), g(y)) 
		&\leq d(T^{\theta_{j+1}} h(y), g(x_j)) + d(g(x_j), g(y)) \\
		&\leq 2^{-j} + d(g(x_j), g(y)) \xrightarrow[j \to \infty]{} 0.
	\end{align*}
	Hence \(y\) belongs to the set in \eqref{h-g}. Since \(U\) was arbitrary, the set is dense.  
\end{proof}

\begin{proposition}
	Let \(X\) be a complete metric space and \(T\) a continuous map on \(X\) with \(\mathcal{AP}\Sigma(T)\neq \emptyset\). If \(\Sigma(T)\) contains a continuous open map \(h \colon X \to X\), then for every \(g \in \mathcal{AP}\Sigma(T)\), the set
	\begin{align*}
		\{\, x \in X : \exists\, \text{B}\in \mathcal{AP}\ \text{such that}\  \lim_{\substack{n\rightarrow \infty \\ n\in \text{B} }}T^{n} h(x)=g(x) \,\}
	\end{align*}
	is residual in \(X\).
\end{proposition}

\begin{proof}
	The argument follows the same line of reasoning as in Proposition \ref{relations}, and in fact it proceeds by induction. Let us highlight only the key step. Fix \(x\in\text{R}(T,h)\cap \mathcal{AP}\text{R}(T,g)\) and \(\delta>0\). Then there exists \(m\in \mathbb{N}\) such that 
	\[
	x\in T^{-m}h(B(x,\delta)).
	\]
	Since \(g\in \mathcal{AP}\Sigma(T)\), for every \(\epsilon>0\) and \(L\in \mathbb{N}\) there exist \(b,r\in \mathbb{N}\) with \(b>m\) such that
	\begin{align*}
		T^{-m}h(B(x, \delta))\cap \left(\bigcap_{\ell=0}^{L} T^{-(b+\ell r)}B(g(x),\epsilon)\right)&\neq \emptyset, \quad\text{or}\\
		h(B(x,\delta))\cap \left(\bigcap_{\ell=0}^{L} T^{-(b-m+\ell r)}B(g(x),\epsilon)\right)&\neq \emptyset.
	\end{align*}
	Set \(a:=b-m\). Then we can choose \(y\in \text{R}(T,h)\cap \mathcal{AP}\text{R}(T,g)\) and \(t>0\) such that \(\overline{B(y,t)}\subset B(x,\delta)\) and 
	\[
	T^{a+\ell r}h(B(y,t))\subset B(g(x), \epsilon) \quad \text{for each } 0\leq \ell\leq L.
	\]
	From here the inductive step proceeds by considering the point \(y\).
\end{proof}

\begin{corollary}
	Let \(X\) be an \(F\)-space and \(T\) a continuous linear operator on \(X\). 
	If \(\Sigma(T)\) contains a surjective operator \(h\), then for every \(g \in \Sigma(T)\) the set 
	\begin{align*}
		\{\, x \in X : \exists\, \theta_n \uparrow \infty \ \text{with}\ T^{\theta_n} h(x) \to g(x)\,\}
	\end{align*}
	is residual in \(X\). If, in addition, \(\mathcal{AP}\Sigma(T)\neq \emptyset\), then for every \(g \in \mathcal{AP}\Sigma(T)\) the set 
	\begin{align*}
		\{\, x \in X : \exists\, \text{B}\in \mathcal{AP}\ \text{such that}\ 
		\lim_{\substack{n\to \infty \\ n\in \text{B}}} T^{n} h(x)=g(x)\,\}
	\end{align*}
	is residual in \(X\).
\end{corollary}

This corollary will play a crucial role in the proof of Corollary~\ref{seq-h-g}, 
which in turn is an essential step towards establishing the theorem stated in 
the introduction concerning the set $\Omega(T)$ and the operator $L_{T}$.

Denote by $\mathrm{Homeo}(X)$ the group of homeomorphisms of $X$ onto itself.

\begin{corollary}
	Let \(T\) be a continuous map on a complete metric space \(X\). 
	If $T$ is recurrent, then $\Sigma(T)\cap \mathrm{Homeo}(X)$ is a subgroup of $\mathrm{Homeo}(X)$.
\end{corollary}

\begin{corollary}\label{group}
	Let $T \in \mathcal{L}(X)$. 
	If $T$ is recurrent, then $\Sigma(T)\cap \mathrm{GL}(X)$ is a subgroup of $\mathrm{GL}(X)$.
\end{corollary}

 \subsection{\(T\bigoplus T\)-problem and dense-lineability}

The relationship between the hypercyclicity of an operator $ T $ and the hypercyclicity of $ T \bigoplus T $ has been a central topic of discussion due to its connections with weak mixing and the hypercyclicity criterion \cite{livro, Bes, Grosse, Herrero}. These properties were shown to be equivalent; however, there exist hypercyclic operators that do not satisfy this condition \cite{bayart2007hypercyclic, rosa}. 

In the work of Costakis et al.~\cite[Question 9.6]{Cos}, the authors raised the question of whether 
$T \oplus T$ is recurrent whenever $T$ is a recurrent operator on a separable infinite-dimensional 
Banach space. More recently, S.~Grivaux, A.~López, and A.~Peris \cite{Sophie} proved that for any (real or complex) 
separable infinite-dimensional Banach space $X$ and any $N \in \mathbb{N}$, there exists 
$T \in \mathcal{L}(X)$ such that $\bigoplus_{i=1}^{N} T : X^{N} \to X^{N}$ is $\mathcal{AP}$-recurrent, 
but for which $\bigoplus_{i=1}^{N+1} T : X^{N+1} \to X^{N+1}$  is not even recurrent..

In the context of linear dynamics, there has been considerable interest in studying the dense-lineability 
of sets of hypercyclic vectors, as well as that of recurrent vectors 
\cite{arbieto2025dense, bernal1999densely, bes1999invariant, bourdon1993invariant, grivaux2025questions, herrero1991limits}. 

Recently, Grivaux et al.~\cite{grivaux2023questions} asked whether $\mathrm{Rec}(T)$ is densely lineable 
whenever $T$ is recurrent, a question that was answered in the negative by A.~López and Q.~Menet 
\cite{LopezM}. They showed that every (real or complex) separable infinite-dimensional Banach space $X$ 
supports an $\mathcal{AP}$-recurrent operator $T : X \to X$ for which the set of recurrent vectors 
$\mathrm{Rec}(T)$, and hence also the set of $\mathcal{AP}$-recurrent vectors 
$\mathcal{AP}\mathrm{Rec}(T)$, is not densely lineable.

 \begin{theorem}\label{Sigma-Tapia}
	Let $ X $ be a separable infinite-dimensional complex Banach space. There exists a \(\mathcal{AP}\)-recurrent operator $ T $ acting on $ X $ such that $ \Sigma(T \bigoplus T) = \emptyset $ and for every $ g \in \Sigma(T) $, the set $ \text{R}(T, g) $ is not dense-lineable.
\end{theorem}

 We recall some aspects of the Augé-Tapia operator. Let $X$ be a separable infinite dimensional complex Banach space. For $x\in X$, the Augé-Tapia operator of type $N$ is given by: 
 \begin{eqnarray*}
 	Tx=Sx+\sum_{k=N+1}^{\infty}\frac{1}{m_{k-1}} g_{k}(\mathbb{P}x)e_{k}.
 \end{eqnarray*} 
 Now, let's explain the components involved. One can find $(e_{k}, e^{*}_{k})_{k\in \mathbb{N}}\subset X\times X^{*}$, as established by the well-known theorem of Ovsepian and Aleksander \cite{Ovsepian}, such that
 $\mathrm{span}\{e_{n}:n\in \mathbb{N}\}$ is dense in $X$, $e^{*}_{n}(e_{m})=1$ for $n=m$ or $0$ for $n\neq m$, $\Vert{e_{n}\Vert}=1$ for each $n\in \mathbb{N}$,  and $\sup_{n\in \mathbb{N}}\Vert{e^{*}_{n}\Vert}<\infty$. Given $N\in \mathbb{N}$, let us denote $\mathbb{P}$ the projection defined on $X$ onto $V:=\mathrm{span}(\{e_{i}\}_{i=1}^{N})$ by $\mathbb{P}(x)=\sum_{i=1}^{N}\langle{e_{i}^{*},x}\rangle e_{i}$. Consider the sequence $(g_{k})_{k} \subset V^{*}$, which is associated with a set $F \subset V$. Here, $F$ and $V\setminus F$ are dense in $V$ such that, for all $x \in F$, $\liminf |\, g_{n}(x)\,| = 0$, and for all $x \notin F$, $\lim |\, g_{n}(x)\,| = +\infty$ \cite[Corollary 2.12]{Tapia}. To define the operator $S$, consider a sequence $(\lambda_{k})_{k\geq 1}$ with $\lambda_{i}=1$ for $i\in \{1,2,\ldots, N\}$ and for $i\geq N+1$, $\lambda_{i}=e^{\frac{i\pi}{m_{i}}}$ where $(m_{k})_{k\geq 1}\subset \mathbb{N}$ be a rapidly increasing sequence so that the following two conditions are satisfied:
 $m_{k}\vert m_{k+1}$ for all $n\in \mathbb{N}$ and $\displaystyle{\sum_{k\geq N+1}\frac{m_{k-2}}{m_{k-1}}\Vert{g_{k}\Vert}<\infty}$.
 For $x \in c_{00}:=\text{span}\{e_{n} : n \geq 1\}$, write $x$ as $x = \sum_{i=1}^{\ell} x_{i}e_{i}$, and define $Sx = \sum_{i=1}^{\ell}\lambda_{i}x_{i}e_{i}$. This establishes $S$ as a bounded operator on $\mathrm{span}\{e_{n} : n \geq 1\}$. Consequently, $S$ can be extended to a bounded operator on $X$.
 
 For $x\in X$ and each positive integer $n$,
 \begin{align*}
 	T^{n}x=S^{n}x+\sum_{k=N+1}^{\infty}\frac{\lambda_{k,n}}{m_{k-1}}g_{k}(\mathbb{P}x)e_{k}\;\; \hbox{where}\;\; \lambda_{k,n}=\sum_{j=0}^{n-1}\lambda_{k}^{j}.
 \end{align*}
 
 This continuous linear operator $T$ satisfies the following property \cite{Auge, Tapia}:
 \begin{eqnarray*}
 	\text{Rec}(T)=\mathbb{P}^{-1}(F)\hspace{0.3cm}\; \text{and}\hspace{0.3cm} \mathbb{P}^{-1}(V\setminus F)=A_{T}=\{x\in X: \lim_{n}\Vert T^{n}x\Vert =\infty\}.
 \end{eqnarray*}

 \begin{proof}[Prof of Theorem \ref{Sigma-Tapia}]
 Consider an Augé--Tapia operator of type~2 on $X$. We claim that $T$ is $\mathcal{AP}$-recurrent.  
 To establish this, we will show that \( \text{Rec}(T)\cap c_{00}\subset \mathcal{AP}\text{Rec}(T)\). In this way, fix an arbitrary $x\in \text{Rec}(T)\cap c_{00}$. Then there exists a strictly increasing sequence of positive integers $(k_{n})_{n}$ such that \( \lvert g_{k_{n}}(\mathbb{P}(x))\rvert \to 0\). Now, consider the open ball centered at $x$, $U:=B(x,\epsilon)$ with $\epsilon>0$.  
 For an arbitrary but fixed \(L \in \mathbb{N}\), it is clear that \( S^{2jm_{k_{n}-1}}x = x\) for each \(1\leq j\leq L\) whenever $n$ is sufficiently large, with which,
 \begin{align*}
 	T^{2jm_{k_{n}-1}}x - x 
 	= \frac{\lambda_{k_{n},2jm_{k_{n}-1}}}{m_{k_{n}-1}}g_{k_{n}}(\mathbb{P}(x))e_{k_{n}}
 	+ \sum_{\ell>k_{n}}\frac{\lambda_{\ell,2jm_{k_{n}-1}}}{m_{\ell-1}}g_{\ell}(\mathbb{P}(x))e_{\ell}
 \end{align*}
Using that $\lvert \lambda_{\ell,t}\rvert \leq t$ for $\ell,t\in\mathbb{N}$, we obtain
\begin{align*}
	\Vert T^{2jm_{k_{n}-1}}x - x\Vert 
	&\leq 2L\,\lvert g_{k_{n}}(\mathbb{P}(x))\rvert
	+ 2L\Vert \mathbb{P}(x)\Vert \sum_{\ell>k_{n}} \frac{m_{\ell-2}}{m_{\ell-1}}\Vert g_{\ell}\Vert 
	\xrightarrow[n\to\infty]{} 0.
\end{align*}
Thus, for $n$ sufficiently large, $\{2jm_{k_{n}-1}:1\leq j\leq L\}\subset N_{T}(x,U)$. Since $L$ was arbitrary, it follows that $\mathcal{N}_{T}(x,U)\in\mathcal{AP}$.

We claim that $\Sigma(T\oplus T)=\emptyset$. Suppose otherwise, that is, $\Sigma(T\oplus T)$ contains some continuous operator $h$ on $X\times X$. Since $\text{R}(T\oplus T,h)$ is residual in $X\times X$, there exist $x,y\in \text{R}(T\oplus T,h)$ with $\text{span}(\mathbb{P}x,\mathbb{P}y)=\mathrm{span}(e_1,e_2)$. Hence, there is a strictly increasing sequence $(\theta_n)_n$ such that both $T^{\theta_n}x$ and $T^{\theta_n}y$ converge. Fix any $q\in F^c$. Then there exists $\xi\in \mathrm{span}(x,y)$ with $\mathbb{P}(\xi)=q$, and therefore $T^{\theta_n}(\xi)$ converges as well. This is a contradiction, since $\xi\in A_T$, and by definition of $A_T$ the sequence $T^{\theta_n}(\xi)$ cannot converge.

To conclude the proof, fix any arbitrary $g\in \Sigma(T)$ and assume that $\text{R}(T,g)$ is densely lineable. Since $X=\text{Rec}(T)\cup A_T$, we have $\text{R}(T,g)\subset \text{Rec}(T)$. Thus, there exists a dense subspace $E\subset X$ contained in $\text{Rec}(T)$. One can notice that
\begin{align*}
	V=\mathbb{P}(E)\subset \mathbb{P}(\text{Rec}(T))=F \subsetneq V,
\end{align*}
which is impossible.
\end{proof}

\begin{remark}
	In \cite{Auge}, Augé constructed the first operator exhibiting wild dynamics. Later, S.~Tapia \cite{Tapia} introduced a variation, now referred to as the Augé--Tapia operator, showing that every infinite-dimensional separable complex Banach space admits a bounded operator $T$ such that $\text{Rec}(T)$ and \(A_{T}\) form a partition of $X$, with both sets being dense \cite[Corollary~3.4]{Tapia}. The operators considered in \cite{Sophie,LopezM} are also modifications of Augé’s original construction. 
\end{remark}





\section{The Framework of \(\Omega(T)\)}\label{section3}

In recent years, certain aspects of linear dynamics have drawn significant interest, particularly those related to products of the same dynamical system. To delve deeper, a key result concerning transitive systems is the well-known theorem of Furstenberg \cite[Theorem 1.51]{Grosse}, which states that if $ T $ is topologically weakly mixing (i.e., $ T \times T $ is transitive), then every $ m $-fold direct sum $ \bigoplus_{i=1}^{m} T \colon X^{m} \to X^{m} $ is transitive. Clearly, the converse also holds.

We say that $ T $ is quasi-rigid if there exists a strictly increasing sequence of positive integers $(\theta_n)$ such that the set 
\[
\{x \in X : \lim T^{\theta_n} x\xrightarrow[n\rightarrow \infty]{} x\}
\]
is dense in $ X $.

\begin{definition}
	Let \(T\) be a continuous map acting on a complete metric space \(X\). 
	We define \(\Omega(T)\) as the set of all continuous maps \(h:X\to X\) for which there exists a strictly increasing sequence of positive integers \((\omega_{n})_{n}\) such that 
	\begin{align}\label{set-Omega}
		\{\, x\in X : T^{\omega_{n}}x \xrightarrow[n\to\infty]{} h(x)\,\}
	\end{align}
	is dense in \(X\). 
\end{definition}

When \(T\) is a continuous linear operator on an \(F\)-space \(X\), we shall, by a slight abuse of notation, denote by \(\Omega(T)\) the subset of \(\mathcal{L}(X)\) consisting of all \(h \in \mathcal{L}(X)\) for which there exists an increasing sequence of positive integers \((\omega_{n})_{n}\) such that the set in \eqref{set-Omega} is dense in \(X\).

It was recently shown that a continuous map \(T\) on a Polish space is quasi-rigid if and only if, for every \(m\in \mathbb{N}\), the direct sum \(\bigoplus_{\ell=1}^{m} T\) is recurrent \cite{Sophie}. 
Equivalently, 
\[
\text{Id}\in \Omega(T) 
\ \Longleftrightarrow \ 
\bigoplus_{\ell=1}^{m}\text{Id}\in \Sigma\!\left(\bigoplus_{\ell=1}^{m} T\right)
\quad \text{for all } m\in \mathbb{N}.
\]

In what follows, we establish connections between the sets \(\Sigma(T)\) and \(\Omega(T)\), highlighting their natural interplay.

\begin{theorem} \label{T-h}
	Let \( X \) be a Polish space, and let \( T \) be a continuous map on \( X \). For a continuous map \( h \) on \( X \), the following assertions are equivalent:
	\begin{enumerate}
			\item \(h\in \Omega(T)\)
		\item For each \( m \in \mathbb{N} \),        
		\begin{align*}
			\bigoplus_{\ell=1}^{m} h \in \Sigma \left(\bigoplus_{\ell=1}^{m} T\right).
		\end{align*}
	\end{enumerate}
\end{theorem}

\begin{proof}
	The implication (1) $\implies$ (2) is immediate. We now show that (2) implies (1), adapting the proof of \cite[Theorem 2.5]{Sophie}. Let $\{U_m\}_{m \in \mathbb{N}}$ be a countable base of open sets for $X$. 
	
	We begin by selecting $a_{1,1} \in U_1$, and note that $h(U_1) \cap B(h(a_{1,1}), 1) \neq \emptyset$. Since $h \in \Sigma(T)$, there exists $n_1 \in \mathbb{N}$ such that $T^{\theta_{1}}(U_1) \cap B(h(a_{1,1}), 1) \neq \emptyset$. By the continuity of $T$, there exists a non-empty open set $A_{1,1} \subset \overline{A_{1,1}} \subset U_1$ with diameter less than $2^{-1}$ such that $T^{\theta_{1}}(A_{1,1}) \subset B(h(a_{1,1}), 1)$.
	
	Next, consider the open sets $A_{1,1} \times U_2$ and $B(h(a_{1,2}), \frac{1}{2}) \times B(h(a_{2,1}), \frac{1}{2})$, where $a_{1,2} \in A_{1,1}$ and $a_{2,1} \in U_2$. By assumption, when $m = 2$, there exists $n_2 > n_1$ such that
	\[
	T^{\theta_{2}}(A_{1,1}) \cap B(h(a_{1,2}), \frac{1}{2}) \neq \emptyset \quad \text{and} \quad T^{\theta_{2}}(U_2) \cap B(h(a_{2,1}), \frac{1}{2}) \neq \emptyset.
	\]
	The continuity of $T$ ensures the existence of two non-empty open sets $A_{1,2} \subset \overline{A_{1,2}} \subset A_{1,1}$ and $A_{2,1} \subset \overline{A_{2,1}} \subset U_2$, each with diameter less than $2^{-2}$, such that
	\[
	T^{\theta_{2}}(A_{1,2}) \subset B(h(a_{1,2}), \frac{1}{2}) \quad \text{and} \quad T^{\theta_{2}}(A_{2,1}) \subset B(h(a_{2,1}), \frac{1}{2}).
	\]
	
	Proceeding inductively, we construct a strictly increasing sequence of positive integers $(\theta_n)$, along with open sets $A_{m,j}$ and points $a_{m,j} \in A_{m,j}$, satisfying the following for each $m$ and $j$:
	\begin{enumerate}
		\item $U_m \supset \overline{A_{m,1}} \supset A_{m,1} \supset \overline{A_{m,2}} \supset A_{m,2} \supset \cdots \supset \overline{A_{m,j}} \supset A_{m,j} \supset \cdots$,
		\item $T^{\theta_{n}}(A_{i,j}) \subset B(h(a_{i,j}), \frac{1}{n})$ for $i + j = n + 1$,
		\item The diameter of $A_{i,j}$ is less than $2^{-(i+j-1)}$.
	\end{enumerate}
	
	By the Cantor intersection theorem, for each $m \in \mathbb{N}$, let $\{y_m\} := \bigcap_{\ell} A_{m,\ell} \subset U_m$. Consequently, the set $\{y_m\}_{m \in \mathbb{N}}$ is dense in $X$, and the points $a_{m,\ell}$ converge to $y_m$ as $\ell \to \infty$.
	
	Fix $m \in \mathbb{N}$ and consider $n > 2m$, we have
	\[
	d(T^{\theta_{n}} y_m, h(y_m)) \leq d(T^{\theta_{n}} y_m, h(a_{m,n-m+1})) + d(h(a_{m,n-m+1}), h(y_m))
	\xrightarrow[n \to \infty]{} 0.
	\]
	This completes the proof.
\end{proof}

Let \(X\) be a separable infinite-dimensional Fréchet space and let \(T \in \mathcal{L}(X)\).
By combining Furstenberg’s Theorem \cite[Theorem~1.51]{Grosse}, Proposition~\ref{hyper-sigma}, and Theorem~\ref{T-h}, we obtain
\begin{align*}
	T \text{ is weakly mixing} 
	& \quad \Longleftrightarrow \quad 
	\Omega(T) = \mathcal{L}(X).
\end{align*}

\begin{definition}
	Let $X$ be a complete metric space and let $T$ be a continuous map on $X$. We denote by $\mathcal{AP}\Omega(T)$ the set of all continuous maps $h: X \to X$ for which there exists $\text{B} \in \mathcal{AP}$ such that
	\begin{align}\label{set-AP.Omega}
		\{x \in X : \lim_{\substack{n \to \infty \\ n \in \text{B}}} T^{\theta_{n}}x = h(x)\}
	\end{align}
	is dense in $X$.
\end{definition}

When \(T\) is a continuous linear operator on an \(F\)-space \(X\), we shall, by a slight abuse of notation, denote by \(\mathcal{AP}\Omega(T)\) the subset of \(\mathcal{L}(X)\) consisting of all \(h \in \mathcal{L}(X)\) for which there exists $\text{B} \in \mathcal{AP}$ such that the set in \eqref{set-AP.Omega} is dense in \(X\).

The next result follows by adapting the proof of Proposition \ref{AP} in light of the arguments used in the proof of Theorem \ref{T-h}.

\begin{theorem}\label{T-APh}
	Let $X$ be a Polish space and $T: X \rightarrow X$ be a continuous map. For a continuous map $h: X \rightarrow X$, the following statements are equivalent:
	\begin{enumerate}
		\item $h \in \mathcal{AP}\Omega(T)$
		\item For each $m \in \mathbb{N}$, 
		\begin{align*}
			\bigoplus_{\ell=1}^{m}h \in \mathcal{AP}\Sigma(\bigoplus_{\ell=1}^{m}T)
		\end{align*}
	\end{enumerate}
\end{theorem}

We may observe that Theorems \ref{T-h} and \ref{T-APh} do not provide an actual insight into the inherent structure of the set \(\Omega(T)\). Our proposal to address aspects of \(\Omega(T)\) relies on the notion of a collection simultaneously approximated by \(T\) (see Definition \ref{def-simul}). Before proceeding further, we present the following motivating examples.

\begin{example}\label{rec-prod}
	Let $X$ be a separable Banach space, $T: X \to X$ a recurrent operator, and $h$ an invertible operator on $X$ with $\|h\| < 1$. If 
	\[
	h\oplus \text{Id} \in \Sigma(T \oplus T),
	\]
	then $T$ is hypercyclic.
\end{example}

\begin{proof}
	By Proposition \ref{group}, we know that \(h^{n} \oplus \mathrm{Id}\) belongs to \(\Sigma(T \oplus T)\) for every \(n \in \mathbb{N}\). Moreover, by Proposition \ref{Sigma-SOT}, it follows that \(\mathbf{0} \oplus \mathrm{Id} \in \Sigma(T \oplus T)\).
	
	Let \(U, V\) be two non-empty open subsets of \(X\). Choose \(p \in U\), \(q \in V\), and \(\varepsilon > 0\) such that \(B(p,\varepsilon) \subset U\) and \(B(q,\varepsilon) \subset V\). Then we may select 
	\[
	(x,y) \in \mathrm{R}(T \oplus T, \mathbf{0}\oplus \mathrm{Id}) \cap \Big(B(p-q,\tfrac{\varepsilon}{2}) \times B(q,\tfrac{\varepsilon}{2})\Big).
	\]
	Hence, there exists a strictly increasing sequence of positive integers \((\theta_{n})_{n}\) such that 
	\[
	T^{\theta_{n}}x \to 0 \quad \text{and} \quad T^{\theta_{n}}y \to y,
	\]
	which implies \(T^{\theta_{n}}(x+y) \to y\). Observe that \(x+y \in U\) and \(y \in V\). Therefore, there exists \(n \in \mathbb{N}\) such that \(	T^{n}U \cap V \neq \emptyset\). 
\end{proof}

\begin{example}
	Let \(X\) be a Banach space and let \(T:X \to X\) be the bounded operator defined by \(T(x)=\lambda x\), where \(\lambda \in \mathbb{T}\) is irrational. Clearly, \(	\Omega(T)=\{\beta\, \mathrm{Id} : \beta \in \mathbb{T}\}.\)
For \(\alpha, \beta \in \mathbb{T}\), if \(	\alpha \mathrm{Id} \oplus \beta \mathrm{Id} \in \Sigma(T \oplus T)\). Then, \(\alpha=\beta\).
\end{example}

Theorem \ref{T-h} highlights the behavior of finite products arising from iterations of the same continuous map.  
The two preceding examples illustrate distinct phenomena involving products of different operators: in the first case, the assumption leads to the hypercyclicity of \(T\), while in the second, it is impossible for \(\Sigma(T\oplus T)\) to contain the product of two distinct operators.  
These insights naturally give rise to the following question: what happens when we extend our analysis to finite products of operators in \(\Omega(T)\)? In order to address this question and further explore the structural features of \(\Omega(T)\), we introduce the notion of a collection simultaneously approximated.

\begin{definition}\label{def-simul}
	Let $ T $ be a continuous map on a Polish space $ X $. We say that a collection of continuous maps $ \{h_\ell\}_{\ell \in J} $ on $ X $ is simultaneously approximated by $ T $ if, for any $ n \in \mathbb{N} $ and any subset $ \{g_i\}_{i=1}^n \subset \{h_\ell\}_{\ell \in J} $, the $ n $-tuple $ (g_1, \ldots, g_n) $ satisfies
	\[
	(g_1, \ldots, g_n) \in \Sigma\left(\bigoplus_{i=1}^n T\right).
	\]
\end{definition}

According to Proposition \ref{relations}, the following two results hold.

\begin{corollary}\label{inver}
	Let $ T $ be a quasi-rigid map on a Polish space $ X $. Let $ M $ and $ N $ be two collections simultaneously approximated by $ T $, with $ N \subset \mathrm{Homeo}(X) $. Then,
	\[
	\{g \circ h^{-1} : g \in M, h \in N\}
	\]
	is also a collection simultaneously approximated by $ T $.
\end{corollary}

\begin{corollary}\label{union}
	Let $ T $ be a quasi-rigid map on a Polish space $ X $. Let $ M $ and $ N $ be two collections simultaneously approximated by $ T $, both containing the identity map, with $ M \subset \mathrm{Homeo}(X) $. Then, $ M \cup N $ is also a collection simultaneously approximated by $ T $.
\end{corollary}

By Zorn's Lemma, every collection simultaneously approximated by $ T $ is contained in some maximal collection simultaneously approximated by $ T $ with respect to inclusion. The same holds for collections simultaneously approximated by $ T $ that are contained in $ \Omega(T) \cap \mathrm{Homeo}(X) $, and similarly for $ \Omega(T) \cap \mathrm{GL}(X) $ when $ X $ is an $ F $-space. 

From Corollary \ref{union}, if \(T\) is a quasi-rigid map on a Polish space \(X\), then the identity map belongs to a unique maximal collection simultaneously approximated by \(T\), contained in \(\Omega(T) \cap \mathrm{Homeo}(X)\).  
Similarly, when \(X\) is a separable \(F\)-space, if \(T\) is a quasi-rigid operator, then the identity operator belongs to a unique maximal collection simultaneously approximated by \(T\), contained in \(\Omega(T) \cap \mathrm{GL}(X)\).  
By a slight abuse of notation, in both cases we shall denote this unique maximal collection by \(\mathcal{G}(T)\).

\begin{proposition}
	Let $ T $ be a quasi-rigid map on a Polish space \(X\). Then $ \mathcal{G}(T) $ is a normal subgroup of $ \Omega(T) \cap \mathrm{Homeo}(X) $. Furthermore, if $ X $ is a separable $ F $-space and $ T $ is a quasi-rigid operator on $ X $, then $ \mathcal{G}(T) $ is a normal subgroup of $ \Omega(T) \cap \mathrm{GL}(X) $.
\end{proposition}

\begin{proof}
	We will show that $ \mathcal{G}(T) $ is a normal subgroup of $ \Omega(T) \cap \mathrm{Homeo}(X) $.	The remaining case follows by a similar argument.

	By Corollary \ref{union}, we have
	\[
	\mathcal{G}(T) = \{h \in \Omega(T) \cap \mathrm{Homeo}(X) : \{\mathrm{Id}, h\} \text{ is simultaneously approximated by } T\}.
	\]
	Moreover, by Corollary \ref{inver}, \(\mathcal{G}(T)\) is a subgroup of \(\Omega(T) \cap \mathrm{Homeo}(X)\). Let \(h \in \Omega(T) \cap \mathrm{Homeo}(X)\). By Corollary \ref{inver}, \(h \mathcal{G}(T) h^{-1}\) is a collection simultaneously approximated by \(T\) and \(\mathrm{Id} \in h \mathcal{G}(T) h^{-1} \subset \Omega(T) \cap \mathrm{Homeo}(X)\). Due to the maximality of \(\mathcal{H}(T)\), it follows that \( h \mathcal{G}(T) h^{-1} \subset \mathcal{G}(T) \) for all \(h \in \Omega(T) \cap \mathrm{Homeo}(X)\). This shows that $ \mathcal{G}(T) $ is normal subgroup in \(\Omega(T) \cap \mathrm{Homeo}(X)\), completing the proof.
\end{proof}

\begin{remark}\label{hG(T)}
	Let $ T $ be a quasi-rigid operator on a separable \(F\)-space $ X $. For $ g \in \Omega(T) $, it follows from Corollary \ref{inver} that the set $ g\mathcal{G}(T) $ is a collection simultaneously approximated by $ T $ that contains $ g $. Moreover, when $ h \in \mathrm{GL}(X) $, we have $ h\mathcal{G}(T) = \mathcal{G}(T)h $, which is the unique maximal collection simultaneously approximated by $ T $ contained in $ \Omega(T) \cap \mathrm{GL}(X) $ that contains $h$.
\end{remark}

\begin{proposition}\label{simult}
	Let $ T $ be a continuous map on a Polish space $ X $ with $ \Omega(T)\neq \emptyset$. Then the following statements hold:
	\begin{enumerate}
		\item The collection $ \{g_i\}_{i \in \mathbb{N}} $ is simultaneously approximated by $ T $ if and only if there exists a strictly increasing sequence of positive integers $ \theta = (\theta_n)_n $ such that, for each $ i \in \mathbb{N} $, the set
		\[
		\{x \in X : \lim_{n \to \infty} T^{\theta_n}(x) = g_i(x)\}
		\]
		is dense in $ X $.
		
		\item If $ \Omega(T) $ contains some open map, then for any two collections $ \{g_i\}_{i \in \mathbb{N}} $ and $ \{h_i\}_{i \in \mathbb{N}} $ that are simultaneously approximated by $ T $, where each $ h_i $ is an open map, there exists a strictly increasing sequence of positive integers $ \theta = (\theta_n)_n $ such that, for each $ i \in \mathbb{N} $, the set
		\[
		\{x \in X : \lim_{n \to \infty} T^{\theta_n} h_i(x) = g_i(x)\}
		\]
		is dense in $ X $.
	\end{enumerate}
\end{proposition}

\begin{proof}
	We provide a sketch of the proof for the second statement; the first follows similarly. To this end, we rely on the following diagram:
	\[
	\begin{array}{cccccccccccccc}
		h_1 & & & & & & & & & & & & & \theta_1 \\
		h_1 & h_1 & & & h_2 & h_2 & & & & & & & & \theta_2 \\
		h_1 & h_1 & h_1 & & h_2 & h_2 & h_2 & & h_3 & h_3 & h_3 & & & \theta_3 \\
		\vdots & \vdots & \vdots & & \vdots & \vdots & \vdots & & \vdots & \vdots & \vdots & & & \vdots \\
		h_1 & h_1 & \cdots & & h_2 & \cdots & & & h_3 & \cdots &  & &  &\theta_n \\
	\end{array}
	\]
	
	For each $ k \in \mathbb{N} $, consider the map
	\[
	S_k := \left(\bigoplus_{\ell=1}^{k} h_1\right) \bigoplus \cdots \bigoplus \left(\bigoplus_{\ell=1}^{k} h_k\right): X^{k^2} \to X^{k^2},
	\]
	which is continuous and open. Similarly, define $ \Lambda_k: X^{k^2} \to X^{k^2} $ by replacing $ h_i $ with $ g_i $ in the definition of $ S_k $.
	
	Note that $ S_k, \Lambda_k \in \Sigma\left(\bigoplus_{\ell=1}^{k^2} T\right) $ for each $ k \in \mathbb{N} $. The approach combines ideas from the proofs of Proposition \ref{relations} and Theorem \ref{T-h}. For each $ i \in \mathbb{N} $, we handle $ h_i $ and $ g_i $. This allows us to choose a strictly increasing sequence $ (\theta_n)_{n \in \mathbb{N}} $ such that
	\[
	\{x \in X : \lim_{n \to \infty} T^{\theta_n} h_i(x) = g_i(x)\}
	\]
	is dense in $ X $.
\end{proof}

\begin{corollary}\label{seq-h-g}
	Let \( T: X \to X \) be a quasi-rigid map on a separable \(F\)-space \( X \). For \( h, g \in \Omega(T) \) such that \( h \) is a surjective operator, there exists a strictly increasing sequence of positive integers \( (\theta_n)_n \) such that the set
	\[
	\{x \in X : T^{\theta_n} h(x) \xrightarrow[n \to \infty]{} g(x)\}
	\]
	is dense in \( X \).
\end{corollary}

\begin{corollary} \label{wm-G(T)}
	Let $ X $ be a separable infinite-dimensional Fréchet space and $ T \in \mathcal{L}(X) $. 
If \(T\) is weak mixing then \(\mathcal{L}(X)\)	 is simultaneously approximated by $ T $ and \(\text{GL}(X)=\mathcal{G}(T)\)
\end{corollary}

For the remainder of this section, \(X\) is a separable infinite-dimensional $ F $-space, and $ T $ is a continuous linear operator acting on \(X\).

\begin{proposition}\label{Omega0-closed}
	The set $\Omega(T)$ is $\mathrm{SOT}$-closed in $\mathcal{L}(X)$.
\end{proposition}

\begin{proof}
	Let $h \in \overline{\Omega(T)}^{\mathrm{SOT}}$. Fix $m \in \mathbb{N}$, and let $U_1, \ldots, U_m, V_1, \ldots, V_m$ be non-empty open subsets of $X$ such that $h(U_i) \cap V_i \neq \emptyset$ for every $i \in \{1, \ldots, m\}$. For each $i$, choose $p_i \in U_i$ and $\varepsilon > 0$ such that $B(h(p_i), \varepsilon) \subset V_i$.
	
	Consider the $\mathrm{SOT}$-basic neighborhood of $h$ defined by
	\[
	\mathcal{N}(h, p_1, \ldots, p_m, \varepsilon) := \{g \in \mathcal{L}(X) : d(g(p_i), h(p_i)) < \varepsilon \text{ for all } i \in \{1, \ldots, m\}\}.
	\]
	Since $h \in \overline{\Omega(T)}^{\mathrm{SOT}}$, there exists $g \in \Omega(T) \cap \mathcal{N}(h, p_1, \ldots, p_m, \varepsilon)$. Hence $g(p_i) \in V_i$ for all $i$, which implies $g(U_i) \cap V_i \neq \emptyset$ for each $i$.
	
	As $g \in \Omega(T)$, there exists a strictly increasing sequence of positive integers $(\theta_n)_{n \in \mathbb{N}}$ such that 
	\[
	T^{-\theta_n}V_i \cap U_i \neq \emptyset \quad \text{for all sufficiently large $n$ and each $i$}.
	\]
	This shows that $\bigoplus_{\ell=1}^m h \in \Sigma(\bigoplus_{\ell=1}^m T)$. Since $m$ was arbitrary, it follows from Theorem \ref{T-h} that $h \in \Omega(T)$.
\end{proof}

Let $ M $ be a subset of a $\mathbb{K}$-vector space $ Y $, where $\mathbb{K}$ is either $\mathbb{R}$ or $\mathbb{C}$. We define $\mathrm{conv}_{\mathbb{K}}(M) \subset Y$ as
\[
\mathrm{conv}_{\mathbb{K}}(M) := \left\{\sum_{i=1}^{m} \alpha_i x_i : x_i \in M, \, \alpha_i \in \mathbb{K} \text{ with } \sum_{i=1}^m \alpha_i = 1 \right\},
\]

Recall that a subset $ M $ of $ Y $ is an affine manifold if for any $ x, y \in M $ and $ t \in \mathbb{K} $, it follows that $ (1-t)x + ty \in M $. It is well known that $ M $ is an affine manifold if and only if $ \mathrm{conv}_{\mathbb{K}}(M) = M $.

\begin{theorem}\label{convex}
	Every maximal collection simultaneously approximated by $ T $ is an SOT-closed affine manifold in $ \mathcal{L}(X) $.
\end{theorem}

\begin{proof}
	This follows directly from Proposition \ref{convex-M} and Proposition \ref{closed}.
\end{proof}

\begin{proposition}\label{convex-M}
If \( M \) is a collection simultaneously approximated by \( T \), then 
\[
\mathrm{conv}_{\mathbb{K}}(M)
\]
is also a collection simultaneously approximated by \( T \).
\end{proposition}

\begin{proof}
	Fix any finite collection \( \{g_1, \ldots, g_m\} \subset \mathrm{conv}_{\mathbb{K}}(M) \). We will show that
	\[
	(g_1, \ldots, g_m) \in \Sigma\left(\bigoplus_{j=1}^m T\right).
	\]
	To this end, consider two non-empty open sets in \( X^m \) given by
	\[
	U_1 \times \cdots \times U_m \quad \text{and} \quad V_1 \times \cdots \times V_m,
	\]
	such that \( g_j(U_j)\subset V_j\) for each \( 1 \leq j \leq m \).
	
	For every \( g_j \in \mathrm{conv}_{\mathbb{K}}(M) \), we have \( g_j = \sum_{k=1}^{\ell(j)} \alpha_{j,k} h_{j,k} \), where \( \{h_{j,k} : 1 \leq k \leq \ell(j)\} \subset M \) and \( \{\alpha_{j,k}\}\subset \mathbb{K} \) with \( \sum_k \alpha_{j,k} = 1 \).
	
	In each open set \( U_j \), we can choose a vector \( q_j \in U_j \). This implies that there exist open sets \( W_{j,k} \) satisfying:
	\begin{itemize}
		\item \( h_{j,k}(q_j) \in W_{j,k} \),
		\item \( \displaystyle{\sum_{k=1}^{\ell(j)} \alpha_{j,k} W_{j,k}} \subset V_j \).
	\end{itemize}
	
	By the continuity of \( h_{j,k} \), there exists an open set \( q_j \in A_j \subset U_j \) such that \( h_{j,k}(A_j) \subset W_{j,k} \). Moreover, since \( \sum_k \alpha_{j,k} q_j = q_j \in A_j \), there exists an open set \( q_j \in B_j \subset A_j \) satisfying \( \sum_k \alpha_{j,k} B_j \subset A_j \).
	
	Notice that \( \{h_{j,k}\}_{j,k} \subset M \) is simultaneously approximated by \( T \), by Proposition \ref{simult}, there exists a strictly increasing sequence of positive integers \( (\theta_n)_{n \in \mathbb{N}} \) such that
	\[
	\{x \in X : T^{\theta_n} x \xrightarrow[n \to \infty]{} h_{j,k}(x)\}
	\]
	is dense for each \( j, k \).
	
	The density of the above set allows us to choose \( x_{j,k} \in B_j \) such that \( \lim_{n \to \infty} T^{\theta_n} x_{j,k} = h_{j,k}(x_{j,k}) \). Therefore, by performing appropriate sums, we have
	\[
	T^{\theta_n} \left( \sum_{k=1}^{\ell(j)} \alpha_{j,k} x_{j,k} \right) \xrightarrow[n \to \infty]{} \sum_{k=1}^{\ell(j)} \alpha_{j,k} h_{j,k}(x_{j,k}).
	\]
	
	Clearly, \( \sum_{k=1}^{\ell(j)} \alpha_{j,k} x_{j,k} \in U_j \) and \( \sum_{k=1}^{\ell(j)} \alpha_{j,k} h_{j,k}(x_{j,k}) \in V_j \) for each \( j \).
	
	Thus, \( \{\theta_n : n \geq n_0\} \subset \{n : U_j \cap T^{-n} V_j \neq \emptyset, \, \forall 1 \leq j \leq m\} \) for some \( n_0 \in \mathbb{N} \). By Proposition \ref{open-sigma}, we conclude the proof.
\end{proof}

\begin{proposition}\label{closed}
If \( M \) is a collection simultaneously approximated by \( T \), then the \(\text{SOT}\)-closure of \( M \) is also a collection simultaneously approximated by \( T \).
\end{proposition}

\begin{proof}
	Fix any finite collection \( \{h_j\}_{j=1}^m \subset \overline{M}^{\text{SOT}} \). We will show that
	\[
	(h_1, \ldots, h_m) \in \Sigma\left(\bigoplus_{j=1}^m T\right).
	\]
	To this end, consider two non-empty open sets in \( X^m \) given by
	\[
	U_1 \times \cdots \times U_m \quad \text{and} \quad V_1 \times \cdots \times V_m,
	\]
	such that \( h_j(U_j)\subset V_j\) for each \( 1 \leq j \leq m \). For each \( j \), choose a vector \( q_j \in U_j \) and \( \epsilon > 0 \) such that \( B(h_j(q_j), \epsilon) \subset V_j \).
	
	Now, consider the following SOT-open neighborhoods for each \( j \):
	\[
	\mathcal{N}(h_j, q_j, \epsilon) = \{f \in \mathcal{L}(X) : d(h_j(q_j), f(q_j)) < \epsilon\}.
	\]
	By the definition of the \(\text{SOT}\)-closure, there exist continuous linear operators \( f_j \in \mathcal{N}(h_j, q_j, \epsilon) \cap M \). Consequently, \( f_j(U_j) \cap V_j \neq \emptyset \) for every \( j \).
	
	Since \( (f_1, \ldots, f_m) \in \Sigma\left(\bigoplus_{j=1}^m T\right) \), it follows that
	\[
	\{n \in \mathbb{N} : U_j \cap T^{-n}(V_j) \neq \emptyset \text{ for all } 1 \leq j \leq m\}
	\]
	is an infinite set. Therefore, by Proposition \ref{open-sigma}, we conclude the proof.
\end{proof}

\begin{theorem}\label{N-G}
	Let \(T\) be a quasi-rigid operator. If $ N $ is a maximal collection simultaneously approximated by \(T\) that contains the identity operator, then:
	\begin{align} \label{maxi}
		\mathcal{G}(T) = N \cap \mathrm{GL}(X) \quad \text{and} \quad \mathrm{conv}_{\mathbb{K}}(\mathcal{G}(T)) \subset N.
	\end{align}
	Moreover, if $ X $ is a Banach space, $ \mathcal{G}(T) $ is locally convex, and the affine manifod:
	\[
	\mathrm{conv}_{\mathbb{K}}(\mathcal{G}(T))
	\]
	is the unique maximal collection simultaneously approximated by $ T $ that contains the identity operator.
\end{theorem}

\begin{proof}
	Let $ N $ be a maximal collection simultaneously approximated by $ T $ that contains the identity operator. By Corollary \ref{union}, we have $ \mathcal{G}(T) \subset N $. Consequently, $ \mathrm{conv}_{\mathbb{K}}(\mathcal{G}(T)) \subset N $ by Theorem \ref{convex}. On the other hand, note that $ \mathrm{Id} \in N \cap \mathrm{GL}(X) $ is a collection simultaneously approximated by $ T $. Due to the maximality of $ \mathcal{G}(T) $, it follows that $ N \cap \mathrm{GL}(X) \subset \mathcal{G}(T) $. Therefore, $ \mathcal{G}(T) = N \cap \mathrm{GL}(X) $.
	
	Now assume $ X $ is a Banach space. Fix any $ g \in \mathcal{G}(T) $. Then there exists $ \epsilon > 0 $ such that $ B(g, \epsilon) \subset \mathrm{GL}(X) $. Consider the open ball $ B(g, r) $ with $ 2r < \epsilon $. Let $ f, h \in \mathcal{G}(T) \cap B(g, r) $. Clearly, for every $ t \in [0, 1] $, $ tf + (1-t)h \in \mathrm{GL}(X) $. Furthermore, $ tf + (1-t)h \in \mathcal{G}(T) $. Thus, $ \mathcal{G}(T) $ is locally convex.
	
	By Proposition \ref{convex-M}, the set $ \mathrm{conv}_{\mathbb{K}}(\mathcal{G}(T)) $ is simultaneously approximated by $ T $ and contains the identity operator. Let $ N $ be any maximal collection simultaneously approximated by $ T $ that contains the identity operator. For any $ g \in N $, since $ \mathrm{GL}(X) $ is an open subset of $ \mathcal{L}(X) $, there exists $ t_0 \neq 1 $ close to $ 1 $ such that $ t_0 \mathrm{Id} + (1-t_0)g \in \mathrm{GL}(X) $, and $ t_0 \mathrm{Id} + (1-t_0)g \in N $ by Theorem \ref{convex}. Consequently, $ t_0 \mathrm{Id} + (1-t_0)g \in \mathcal{G}(T) $, which implies $ g \in \mathrm{conv}_{\mathbb{K}}(\mathcal{G}(T)) $. Therefore, $ N \subseteq \mathrm{conv}_{\mathbb{K}}(\mathcal{G}(T)) $, and by maximality, $ N = \mathrm{conv}_{\mathbb{K}}(\mathcal{G}(T)) $. Thus, there exists a unique maximal collection simultaneously approximated by $ T $ that contains the identity operator, given by $ \mathrm{conv}_{\mathbb{K}}(\mathcal{G}(T)) $.
\end{proof}

\begin{theorem}
	Let \( X \) be an infinite-dimensional separable Fréchet or Banach space. Then, the set
	\[
	\{T \in \mathcal{L}(X) : T \text{ is quasi-rigid and } \{\mathrm{Id}\} \subsetneq \mathcal{G}(T) \subsetneq \mathrm{GL}(X)\}
	\]
	is \(\text{SOT}\)-dense in \(\mathcal{L}(X)\).
\end{theorem}

\begin{proof}
	Let \( S \in \mathcal{L}(X) \), and consider a \(\text{SOT}\)-open basic neighborhood of \( S \):
	\[
	\mathcal{N}(S, a_1, \ldots, a_m, \epsilon) := \{\Lambda \in \mathcal{L}(X) : d(\Lambda(a_i), S(a_i)) < \epsilon \text{ for all } 1 \leq i \leq m\},
	\]
	where \( \{a_1, \ldots, a_m\} \subset X \) are linearly independent vectors and \( \epsilon > 0 \).
	
	We can choose vectors \( \{a_{m+1}, \ldots, a_{2m}\} \subset X \) such that \( \{a_i : 1 \leq i \leq 2m\} \) are linearly independent and \( d(a_{i+m}, S(a_i)) < \epsilon \) for each \( 1 \leq i \leq m \).
	
	Let \( N \) be the finite-dimensional subspace generated by \( \{a_i : 1 \leq i \leq 2m\} \). By the Hahn-Banach Theorem, there exist \( \{a_j^*\} \subset X^* \) such that \( a_j^*(a_i) = \delta_{i,j} \) for \( i, j = 1, \ldots, 2m \).
	
	Denote by \( M \) the topological complement of \( N \), so that \( X = N \oplus M \). By the Ansari-Bernal Theorem, there exists a weakly mixing operator \( A : M \to M \). Define \( T \in \mathcal{L}(X) \) as follows:
	\[
	T : N \oplus M \to N \oplus M, \quad (x, y) \mapsto \left(\sum_{i=1}^m a_i^*(x)a_{i+m} + \sum_{i=m+1}^{2m} a_i^*(x)a_{i-m}, A(y)\right).
	\]
	
	Clearly, \( T \) is quasi-rigid but not weakly mixing, since \( T^2 = \mathrm{Id}|_N \times A^2 \). Moreover, using Corollary \ref{wm-G(T)}, it is straightforward to verify that
	\[
	\mathrm{Id}|_N \times \mathrm{GL}(M)=\mathcal{G}(T).
	\]
	Thus, \( \{\mathrm{Id}\} \subsetneq \mathcal{G}(T)\subsetneq \mathrm{GL}(X) \), and \( T \in \mathcal{N}(S, a_1, \ldots, a_m, \epsilon) \). This completes the proof.
\end{proof}

To conclude this section, we provide two results on hypercyclic operators and formulate open problems concerning $\Omega(T)$ in the framework of the left multiplication operator $L_T$.

\begin{remark}
	With a slight abuse of the classical notion of a cone, we will refer to any non-empty subset $ C $ of a $\mathbb{K}$-vector space that is invariant under scalar multiplication by elements of $\mathbb{K}$ as a \emph{cone}. That is, $ \lambda C \subseteq C $ for all $ \lambda \in \mathbb{K} $.
\end{remark}

\begin{proposition}\label{cone}
	Let $ T $ be a hypercyclic operator acting on a separable infinite-dimensional $ F $-space. Then $ \Omega(T) $ is a cone and contains an infinite-dimensional subspace of $ \mathcal{L}(X) $.
\end{proposition}

\begin{proof}
	Fix $ h \in \Omega(T) $. Let $ \alpha \in \mathbb{K} $ and let $ x $ be a hypercyclic vector for $ T $. We can find a strictly increasing sequence of positive integers $ (\theta_n)_n $ such that $ T^{\theta_n}x \to \alpha x $ as $ n \to \infty $. Therefore, the orbit of $ x $, $ \{T^m x\}_{m \in \mathbb{N}} $, is contained in 
	\[
	\{y \in X : \lim_{n \to \infty} T^{\theta_n}y = \alpha y\}.
	\]
	This implies that $ \mathbb{K} \cdot \mathrm{Id} \subset \Omega(T) $. Since the zero operator belongs to $ \Omega(T) $ and by Corollary \ref{inver}, it follows that $ \mathbb{K} \cdot h \subset \Omega(T) $. Thus, $ \Omega(T) $ is a cone.
	
	We now claim that $ \text{span}\{T^n : n \geq 0\} \subset \Omega(T) $. Fix any polynomial $ P \in \mathbb{K}[t] $ and let $ q $ be a hypercyclic vector for $ T $. Then, there exists a strictly increasing sequence of positive integers $ (\theta_n)_n $ such that $ \lim_{n \to \infty} T^{\theta_n}q = P(T)q $. It is straightforward to verify that the orbit $ \{T^m q\}_{m \in \mathbb{N}} $ is contained in 
	\[
	\{z \in X : \lim_{n \to \infty} T^{\theta_n}z = P(T)z\}.
	\]
	This concludes the proof.
\end{proof}

Recalling Remark \ref{hG(T)}, when considering a hypercyclic operator $ T $ and $ h = \textbf{0} \in \Omega(T) $, it is clear that $ h\mathcal{G}(T) = \{\textbf{0}\} $ does not provide additional information beyond what is already known. However, it is possible to find hypercyclic operators that are not weakly mixing, for which the zero operator is contained in an infinite-dimensional subspace that is simultaneously approximated by said operator. This is illustrated in the following example.

\begin{example}
	Let $ H $ be a separable infinite-dimensional complex Hilbert space. There exists a continuous linear operator $ T \in \mathcal{L}(H) $ that is hypercyclic but not weakly mixing, such that the zero operator is contained in an infinite-dimensional subspace of $ \mathcal{L}(H) $ that is simultaneously approximated by $ T $.
\end{example}

\begin{proof}
	We may assume that $ H := \ell^2(\mathbb{Z}) $ with the canonical orthonormal basis $ (e_n)_{n \in \mathbb{Z}} $. Now consider $ M := \overline{\text{span}}(e_i : i < 0) $ and $ N := \overline{\text{span}}(e_i : i \geq 0) $, so that $ M \oplus N = H $. According to \cite[Corollary 4.15]{livro}, there exists a bounded operator $ S : M \to M $ that is hypercyclic but not weakly mixing. Additionally, let $ \Lambda : N \to N $ be a bounded operator satisfying the hypercyclicity criterion for the sequence $ (n_k) := (k) $.
	
	Define $ T \in \mathcal{L}(H) $ by
	\[
	T : M \oplus N \to M \oplus N, \quad (x, y) \mapsto (Sx, \Lambda y).
	\]
	Clearly, $ T $ is hypercyclic but not weakly mixing. We now show that the infinite-dimensional subspace 
	$ 0|_M \times \mathrm{span}\{\Lambda^m : m \geq 0\} $ is simultaneously approximated by $ T $. To this end, consider a finite collection of polynomials in $ \Lambda $, 	$\{P_{\ell}(\Lambda)\}_{\ell=1}^{m}$ with $P_{\ell} \in \mathbb{C}[t]$ for each $\ell$. We claim that
	\[
	\{0,\, 0|_M \times P_{1}(T),\, \dots,\, 0|_M \times P_{m}(T)\}
	\]
	is simultaneously approximated by $ T $.
	
	Let $ x_0 \in M $ be a hypercyclic vector of $ S $. Then there exists a strictly increasing sequence of positive integers $ (\omega_n)_{n \in \mathbb{N}} $ such that $ S^{\omega_n}x_0 \to 0 \in M $. On the other hand, since $ \Lambda $ satisfies the hypercyclicity criterion for the sequence $ (n_k) := (k) $, it is possible to find $ \{y_\ell\}_{\ell=1}^m \subset N $, where each $ y_\ell $ is hypercyclic for $ \Lambda $, and a subsequence $ (\psi_n) $ of $ (\omega_n) $ such that $ \lim_{n}\Lambda^{\psi_n}y_\ell = P_{\ell}(\Lambda)y_\ell $ for each $ \ell \in \{1, \dots, m\} $.
	
	One can observe that
	\[
	\mathrm{span}\{S^i x_0 : i \geq 0\} \times \mathrm{span}\{\Lambda^i y_\ell : i \geq 0\} \subset \{(x, y) \in H : T^{\psi_n}(x, y) \to (0, P_{i}(\Lambda)y)\}
	\]
	is a dense subspace for each $ \ell \in \{1, \dots, m\} $. By Proposition \ref{simult}, it follows that 
	\[
	\{0, 0|_M \times P_{1}(\Lambda), \dots, 0|_M \times P_{m}(\Lambda)\}
	\]
	is simultaneously approximated by $ T $.
\end{proof}


Let $T$ be a bounded operator on a separable Banach space $X$. We define the left-multiplication operator $L_{T}$ on the operator algebra $\mathcal{L}(X)$ as follows:
\begin{eqnarray*}
	\text{L}_{T} : \mathcal{L}(X) &\longrightarrow& \mathcal{L}(X), \\
	{}    S &\longmapsto& TS
\end{eqnarray*}
Note that $\Vert \text{L}_{T} \Vert = \Vert T \Vert$. Although $(\mathcal{L}(X), \Vert \cdot \Vert)$ in general is not separable when $X$ is a separable infinite-dimensional Banach space, it is known that $\mathcal{L}(X)$ is $\mathrm{SOT}$-separable \cite{Chan}.

Recall that hypercyclic operators are nowhere dense under the operator norm topology \cite{Wu}. In contrast, \cite{Chan-densi} shows that in a separable infinite-dimensional Hilbert space $H$, hypercyclic operators are $\mathrm{SOT}$-dense in $\mathcal{L}(H)$. This result has been extended to separable infinite-dimensional Fréchet spaces \cite{Bes-Chan}.

We say that $\text{L}_{T}$ is $\mathrm{SOT}$-hypercyclic in $\mathcal{L}(X)$ if there exists a operator $S\in \mathcal{L}(X)$ such that $\{T^{n}S\}_{n\geq 0}$ is $\mathrm{SOT}$-dense in $\mathcal{L}(X)$. Such an operator $S$ is then called $\mathrm{SOT}$-hypercyclic for $\text{L}_{T}$.

A relevant result in this context ensures that, on a separable infinite-dimensional Banach space, \( \text{L}_T \) is \(\mathrm{SOT}\)-hypercyclic if and only if \( T \) satisfies the Hypercyclicity Criterion. For a proof, we refer the reader to \cite{livro, chan1999hypercyclicity, Chan, Grosse}.

It is worth noting that \((\mathcal{L}(X), \mathrm{SOT})\) is not a Baire space. In \cite{livro}, to prove that \( \text{L}_T \) is \(\mathrm{SOT}\)-hypercyclic, the restriction of \( \text{L}_T \) to separable Banach space $(\mathcal{FIN}, \Vert{\cdot\Vert})$ is studied, where \(\mathcal{FIN}\) is the closure of the set of finite rank operators in \(\mathcal{L}(X)\) under the operator norm topology. Specifically, under these conditions, \((\text{L}_T)|_{\mathcal{FIN}}: (\mathcal{FIN}, \|\cdot\|) \to (\mathcal{FIN}, \|\cdot\|)\) is hypercyclic.

Proposition \ref{Omega0-closed} ensures that $ \Omega(T) $ is $ \mathrm{SOT} $-closed in $ \mathcal{L}(X) $. Moreover, it is not difficult to verify that $ \Omega(T) $ is $ \text{L}_T $-invariant.

\begin{theorem}\label{q-transi}
	Let $X$ be a separable Fréchet space, and let $T$ be a quasi-rigid operator on $X$. Then
	\[
	\Omega(T) \subset \overline{\bigcup_{n \in \mathbb{N}} \text{L}_{T}^n(A)}^{\mathrm{SOT}}
	\]
	for every open set $A \subset (\mathcal{L}(X), \mathrm{SOT})$ that contains a surjective operator in $\Omega(T)$.
\end{theorem}

\begin{proof}
	Let $A$ be a basic $\mathrm{SOT}$-open neighborhood of some surjective operator $h \in \Omega(T)$, given by
	\[
	A := \mathcal{N}(h, x_{1}, \ldots, x_{k}, \varepsilon),
	\]
	where $\varepsilon > 0$ and $\{x_i\}_{i=1}^k$ is a finite linearly independent set in $X$.  
	Now fix an arbitrary $g \in \Omega(T)$ and consider a basic $\mathrm{SOT}$-open neighborhood of $g$, namely
	\[
	W := \mathcal{N}(g, y_{1}, \ldots, y_{m}, \delta),
	\]
	where $\delta > 0$ and $\{y_j\}_{j=1}^m$ is a finite linearly independent set in $X$. Without loss of generality, we may assume that $\{x_1, \ldots, x_k, y_1, \ldots, y_m\}$ is linearly independent.  
	
	According to Proposition \ref{seq-h-g}, there exists a strictly increasing sequence of positive integers $(\theta_n)_{n \in \mathbb{N}}$ such that the set
	\[
	G := \{x \in X : \lim_{n \to \infty} T^{\theta_n}h(x) = g(x)\}
	\]
	is dense in $X$.  
	
	Hence we can choose $\{x_i'\}_{i=1}^k, \{y_j'\}_{j=1}^m \subset G$ such that
	\[
	d(h(x_i), h(x_i')) < \varepsilon 
	\quad \text{and} \quad 
	d(g(y_j), g(y_j')) < \delta 
	\quad \text{for all } i,j.
	\]
	Furthermore, there exists an operator $S \in \mathcal{L}(X)$ satisfying 
	$S(x_i) = h(x_i')$ and $S(y_j) = h(y_j')$ for all $i, j$. Clearly $S \in A$. Moreover, for sufficiently large $n$, we have $\text{L}_T^{\theta_n} S \in W$. This completes the proof.
\end{proof}

\begin{corollary}
	Let $X$ be a separable Fréchet space, and let $T$ be a quasi-rigid operator on $X$. 
	Assume that $\mathcal{AP}\Omega(T)\neq \emptyset$. 
	Then for any non-empty open sets $U, V \subset (\mathcal{L}(X), \mathrm{SOT})$ such that $U$ contains a surjective operator from $\Omega(T)$ and $V \cap \mathcal{AP}\Omega(T) \neq \emptyset$, one has
	\[
	\{\, n \in \mathbb{N} : \text{L}_T^n(U) \cap V \neq \emptyset \,\} \in \mathcal{AP}.
	\]
\end{corollary}

We conclude this section by presenting the following open problems.

\begin{question}
	Let $ X $ be a separable Banach space and $ T \in \mathcal{L}(X) $. 
	Are the sets $ \Omega(T) $ and $ \mathcal{AP}\Omega(T) $ $ \mathrm{SOT} $-separable in $ \mathcal{L}(X) $?
\end{question}

An affirmative answer to this problem would have significant consequences for dense-lineability and spaceability, as will become evident in the next section through Theorems \ref{O-c.d.l} and \ref{ava-partial}.

In light of these considerations, together with Theorem \ref{q-transi}, we propose the following problem concerning the structure of $ \Omega(T) $ and the dynamical behavior of the system $ (\Omega(T), \text{L}_T) $.

\begin{question}
	Let $ X $ be a separable Banach space and $ T \in \mathcal{L}(X) $ a quasi-rigid operator. 
	Does there exist an operator $ \Lambda \in \mathcal{L}(X) $ such that
	\[
	\Omega(T) \subset \overline{\{ \text{L}_{T}^n \Lambda : n \in \mathbb{N} \}}^{\mathrm{SOT}}?
	\]
\end{question}


\section{Spaceability within $\Omega(T)$}\label{section4}	

On a separable infinite-dimensional Banach space $X$, the existence of hypercyclic and recurrent subspaces for weakly mixing and quasi-rigid operators, respectively, is characterized by the analytical condition that the essential spectrum of $T$ intersects the closed unit disk $\overline{\mathbb{D}}$. Equivalently, this holds if and only if there exists an infinite-dimensional closed subspace $E \subset X$ together with an increasing sequence $(\theta_{n})_{n}$ such that
\[
\sup_{n}\,\|T^{\theta_{n}}|_{E}\| < \infty.
\]

A structural object capturing the dynamical features of an operator $T$ is the set $\Omega(T)$. For instance, $T$ is quasi-rigid if and only if $\mathrm{Id} \in \Omega(T)$, and $T$ is weakly mixing if and only if $\Omega(T) = \mathcal{L}(X)$, as established in the previous section. In what follows, we study spaceability within $\Omega(T)$.

In \cite{Montes}, A. Montes-Rodríguez established sufficient conditions for the existence of hypercyclic subspaces. Specifically, let \(X\) be a (real or complex) separable Banach space and let \(T \in \mathcal{L}(X)\). If there exists an increasing sequence of integers \((k_{n})_{n \in \mathbb{N}}\) such that \(T\) satisfies the Hypercyclicity Criterion with respect to \((k_{n})_{n \in \mathbb{N}}\), and if there is an infinite-dimensional closed subspace \(E \subset X\) such that \(T^{k_{n}}x \to 0\) for every \(x \in E\), then \(T\) admits a hypercyclic subspace.  

In the context of quasi-rigid operators, A. López \cite{Lopez} established sufficient conditions for the existence of recurrent subspaces, which can be stated as follows:

\begin{theorem}[\cite{Lopez}]\label{condi-lo}
	Let $ X $ be a (real or complex) separable Banach space and let $ T \in \mathcal{L}(X) $. Assume there exists an increasing sequence of integers $ (k_n)_{n \in \mathbb{N}} $ such that:
	\begin{enumerate}
		\item[i)] The set $ D := \{x \in X : T^{k_n}x \xrightarrow[n \to \infty]{} x\} $ is dense in $ X $,
		\item[ii)]  There exists a non-increasing sequence $ (E_n)_{n \in \mathbb{N}} $ of infinite-dimensional closed subspaces of $ X $ such that 
		\[
		\sup_{n \in \mathbb{N}} \|T^{k_n}|_{E_n}\| < \infty.
		\]
	\end{enumerate}
	Then $ T $ has a recurrent subspace. In particular, there exists an infinite-dimensional closed subspace $ F $ and a subsequence $ (\ell_n)_{n \in \mathbb{N}} $ of $ (k_n)_{n \in \mathbb{N}} $ such that $ T^{\ell_n}x \to x $ for all $ x \in F $.
\end{theorem}

Two key notions related to recurrent and hypercyclic subspaces are the essential spectrum, $\sigma_e(T)$, and the left-essential spectrum, $\sigma_{\ell e}(T)$. More precisely, $\lambda \in \sigma_e(T)$ if and only if $T-\lambda$ is not a Fredholm operator. We say that $S \in \mathcal{L}(X)$ is Fredholm if $\text{Ran}(S)$ is closed, $\text{dim} \,\text{ker}(S) < \infty$, and $\text{codim}\,\text{Ran}(S) < \infty$. On the other hand, $\lambda \in \sigma_{\ell e}(T)$ if and only if $T - \lambda$ is not a left-Fredholm operator. We say that $S \in \mathcal{L}(X)$ is left-Fredholm if $\text{Ran}(S)$ is closed and $\text{dim}\, \text{ker}(S) < \infty$.

The connections between the existence of hypercyclic subspaces and the essential spectrum are reflected in the following result due to M. González, F. León Saavedra and A. Montes Rodríguez

\begin{theorem}[\cite{Gonzales}]\label{subhyper-equi}
	Let $X$ be a separable infinite dimensional complex Banach space, and let $T\in \mathcal{L}(T)$.  Suppose  that  \(T\) satisfies the Hypercyclic criterion. Then  the  following  conditions are  equivalent:
	\begin{enumerate}
		\item $T$ has a Hypercyclic subspace.
		\item There exists an infinite dimensional closed subspace $E\subset X$ and an increasing sequence of integers $(\theta_{n})_{n}$ such that $T^{\theta_{n}}x\xrightarrow[n\rightarrow \infty]{} 0$ for all $x\in E$.
		\item There exists an infinite-dimensional closed subspace $E\subset X$ such that and an increasing sequence of integers $(\theta_{n})_{n}$ such that $\sup_{n}\|T^{\theta_{n}}|_{E}\|<\infty$
		\item the  essential  spectrum  of  \(T\) intersects  the  closed  unit  disk.
	\end{enumerate}
\end{theorem}

Analogously, A. López established connections between the existence of recurrent subspaces and the essential spectrum in the setting of quasi-rigid operators.

\begin{theorem}[\cite{Lopez}]\label{subrec-equi}
	Let \(X\) be a separable infinite dimensional complex Banach space and let \(T\in \mathcal{L}(X)\). If \(T\) is quasi-rigid, then
	the following statements are equivalent:
	\begin{enumerate}
		\item \(T\) has a recurrent subspace;
		\item there exists an infinite-dimensional closed subspace \(E\subset X\) and an increasing sequence of integers
		\((\theta_{n})_{n}\) such that \(T^{\theta_{n}}x \xrightarrow[n\rightarrow \infty]{} x\) for all \(x\in E\);
		\item There exists an infinite-dimensional closed subspace $E\subset X$ such that and an increasing sequence of integers $(\theta_{n})_{n}$ such that $\sup_{n}\|T^{\theta_{n}}|_{E}\|<\infty$;
		\item the  essential  spectrum  of  \(T\) intersects  the  closed  unit  disk.
	\end{enumerate}
\end{theorem}

López established the preceding result for both complex and real Banach spaces. In the real case, the last condition is reformulated as requiring that the essential spectrum of the complexification of $T$ intersects the closed unit disk. Furthermore, he proved that these equivalences remain valid even when the Banach space $X$ is not separable.

At this stage, we aim to establish sufficient conditions for common spaceability by examining a countable collection of operators in $\Omega(T)$. To this end, let us first reflect on the hypotheses of Theorem \ref{condi-lo}: on the one hand, we require an analogue of condition (i) adapted to a countable family of operators, while on the other hand, we need a condition of type (ii) for such a family. 

Our first requirement is addressed by a recent result on common dense-lineability due to A. Arbieto and the first author \cite{arbieto2025dense}, which we now state.

\begin{proposition}\label{seq}
	Let $ X $ be a separable infinite-dimensional $ F $-space. Let $ T \in \mathcal{L}(X) $, and suppose that $ \Omega(T) $ is non-empty. Consider any countable collection $ \{h_i\}_{i \in \mathbb{N}} \subset \Omega(T) $. Then, for each $ i \in \mathbb{N} $, there exists a strictly increasing sequence of positive integers $ (\theta_{n,i})_n $ such that 
	\[
	\bigcap_{i \in \mathbb{N}} \{x \in X : T^{\theta_{n,i}}x \xrightarrow[n \to \infty]{} h_i(x)\}
	\]
	is a dense subspace of $ X $, 
\end{proposition}

\begin{theorem}\label{gen.sub}
	Let $ X $ be a (real or complex) separable infinite-dimensional Banach space, and let $ T \in \mathcal{L}(X) $. Suppose that $ \Omega(T) $ is non-empty, and consider $ \{g_\ell\}_{\ell \in \mathbb{N}} \subset \Omega(T) $. If for each $ \ell \in \mathbb{N} $, there exist strictly increasing sequences of positive integers $ (\theta_{j,\ell})_{j \in \mathbb{N}} $ such that:
	\begin{itemize}
		\item The vector subspace \(\displaystyle{	Z := \bigcap_{\ell\in \mathbb{N}} \{x \in X : T^{\theta_{j,\ell}} x \xrightarrow[j \to \infty]{} g_\ell(x)\}}\) is dense in $ X $.
		
		\item  There exists a non-increasing sequence \((E_n)_{n \in \mathbb{N}}\) of infinite-dimensional closed subspaces of \(X\) such that,
		\[
		\sup_{n\geq 1} \max_{\substack{j,\ell \geq 1 \\ j + \ell = n+1}} \|T^{\theta_{j,\ell}}|_{E_{n}}\| < \infty.
		\]
	\end{itemize}
	Then there exists an infinite-dimensional closed subspace $ M $ and for each $ \ell \in \mathbb{N} $, there exists a subsequence $ (\psi_{j,\ell})_{n \in \mathbb{N}} $ of $ (\theta_{j,\ell})_{j \in \mathbb{N}} $ such that
	\[
	T^{\psi_{j,\ell}}x \xrightarrow[j \to \infty]{} g_\ell(x), \quad \forall x \in M, \; \forall \ell \in \mathbb{N}.
	\]
\end{theorem}

The approach to the proof is based on basic sequence techniques, similar to those used in the proof of Theorem \ref{condi-lo} and in \cite{Montes}. For aspects related to Schauder bases and basic sequences, we refer to the book \cite{Megginson}. Following López’s construction in the proof of Theorem \ref{condi-lo}, we adapt his technique to deal with arbitrary countable collections of operators.

Let $(E_{n})_{n}$ be the collection of subspaces appearing in the second condition of the hypothesis of the previous theorem. By Mazur’s construction \cite[Lemma C.1.1]{livro}, there exists a normalized basic sequence $(e_n)_{n \in \mathbb{N}}$ with $e_n \in E_n$, which constitutes a Schauder basis for $E := \overline{\mathrm{span}}\{e_n : n \in \mathbb{N}\}$. The same property holds for any subsequence of a normalized basic sequence. Denote by $(e_n^*)_{n} \subset E^*$ the corresponding sequence of coordinate functionals on $E$, characterized by $\langle e_m^*, \sum_k \alpha_k e_k \rangle = \alpha_m$ for each $m \in \mathbb{N}$. Moreover, since $(e_n)_{n}$ is normalized, it follows that $\sup_n \|e_n^*\| < \infty$ \cite{Megginson}, and we define $C := 1 + \sup_{m} \|e_m^*\|$.

\begin{proof}[Proof of Theorem \ref{gen.sub}]
We claim that for each $\ell \in \mathbb{N}$, there exists a subsequence $(\psi_{j, \ell})_j$ of $(\theta_{j, \ell})_j$, a strictly increasing sequence of positive integers $(\omega_n)_n$, and a sequence of vectors $(p_{\omega_n}) \subset Z$. Denoting $(q_{\omega_n})_n := (p_{\omega_n} - e_{\omega_n})_n$, these satisfy the following conditions:
\begin{itemize}
\item[i)] \(\Vert{p_{\omega_{n}}-e_{\omega_{n}}\Vert}=\Vert{q_{\omega_{n}}\Vert}<2^{-(n+1)}C^{-1}\) for each \(n\in \mathbb{N}\)
\item[ii)] \(\Vert{T^{\psi_{j,\ell}}(q_{\omega_{n}})\Vert}<2^{-(j+n)}\) for each \(j,\ell\in \mathbb{N}\) and \(n>j+\ell\) 
\item[iii)] \(\Vert{T^{\psi_{j, \ell}}(p_{\omega_{n}})-g_{\ell}(p_{\omega_{n}})\Vert}<2^{-(j+n)}\) for each \(j,\ell\in \mathbb{N}\) and \(1\leq n \leq j+\ell\).
\end{itemize}
The proof of the above statement proceeds by induction. Assume that we have constructed $(\psi_{j,\ell})_{j+\ell\leq i}$, $(\omega_n)_{n=1}^{i}$, and $(p_{\omega_n})_{n=1}^{i}$ satisfying the following conditions:

\begin{itemize}
	\item[(a)] \(\Vert{p_{\omega_{n}}-e_{\omega_{n}}\Vert}=\Vert{q_{\omega_{n}}\Vert}<2^{-(n+1)}C^{-1}\) for each \(n\leq i\),
	\item[(b)] \(\Vert T^{\psi_{j, \ell}}q_{\omega_{n}}\Vert<2^{-(j+n)}\) for each \(2\leq j+\ell<i\), \(i\geq n>j+\ell\),
	\item[(c)] \(\Vert T^{\psi_{j, \ell}}p_{\omega_{n}}-g_{\ell}(p_{\omega_{n}})\Vert<2^{-(j+n)}\) for each \(2\leq j+\ell\leq i\), \(1\leq n\leq j+\ell\) .
\end{itemize}
One can observe that, by the continuity of $ T $, there exists $ \epsilon > 0 $ such that
\begin{align}\label{desi}
	\|T^{\psi_{j, \ell}}y\| < \frac{1}{2^{j+(i+1)}} \quad \text{for each } 2 \leq j+\ell \leq i, \text{ and } y \in X \text{ with } \|y\| < \epsilon.
\end{align}
We proceed to select the next space from the collection $ \{E_n : n > \omega(i)\} $. To this end, consider the set $ \{\psi_{j,\ell} : j + \ell = i\} $, and fix any element from this set, for instance, $ \psi_{j,\ell} $. It is clear that $ \psi_{j,\ell} = \theta_{\zeta, \ell} $ for some $ \zeta \in \mathbb{N} $. We can choose $ \omega_{i+1} > \omega_i $ such that 
\[
\omega_{i+1} > \max\{\zeta + \ell : \psi_{j,\ell} = \theta_{\zeta, \ell}, \, j + \ell = i\},
\]
ensuring the inclusion $ E_{\omega_{i+1}} \subset E_{\zeta + \ell} $. Consequently, for $ j + \ell = i $,
\begin{align*}
	\|T^{\psi_{j,\ell}}|_{E_{\omega_{i+1}}}\| \leq \|T^{\theta_{\zeta,\ell}}|_{E_{\zeta+\ell}}\|\leq 	\sup_{n\geq 1} \max_{j+\ell=n+1}\|T^{\theta_{j,\ell}}|_{E_{n}}\|.
\end{align*}
Now, we choose $ p_{\omega_{i+1}} \in Z $ such that
\begin{align*}
	\|p_{\omega_{i+1}} - e_{\omega_{i+1}}\| < \min\left\{\frac{1}{2^{i+2}C}, \epsilon\right\}.
\end{align*}
This ensures that condition (a) holds for $ i+1 $. Furthermore, from the inequality above and (\ref{desi}), we obtain
\begin{align*}
	\|T^{\psi_{j,\ell}}q_{\omega_{i+1}}\| < \frac{1}{2^{j+(i+1)}} \quad \text{for each } 2 \leq j+\ell < i+1
\end{align*}
thus verifying condition (b) for $ i+1 $.

For each $ \ell \in \{1, \ldots, i\} $, we choose $ \psi_{1,i} \in \{\theta_{n,i} : n \in \mathbb{N}\} $, and for $ 1 \leq \ell < i $, we select $ \psi_{i-\ell+1,\ell} \in \{\theta_{n,\ell} : n \in \mathbb{N}\} $ as follows:
\begin{align*}
	\begin{matrix}
		\ell=1 & \psi_{1,1} & \psi_{2,1} & \cdots & \psi_{i-2,1} & \psi_{i-1,1} & \textcolor{blue}{\psi_{i,1}} \\
		\ell=2 & \psi_{1,2} & \psi_{2,2} & \cdots & \psi_{i-2,2} & \textcolor{blue}{\psi_{i-1,2}} \\
		\vdots & \vdots & \vdots & \ddots & {} & {} \\
		\ell=i-1 & \psi_{1,i-1} & \textcolor{blue}{\psi_{2,i-1}} & {} & {} & {} \\
		\ell=i & \textcolor{blue}{\psi_{1,i}}
	\end{matrix}
\end{align*}
The selection is made such that $ \psi_{i-\ell+1,\ell} > \psi_{i-\ell,\ell} $ and, additionally, the following condition is satisfied:
\begin{align*}
	\|T^{\psi_{i-\ell+1,\ell}}p_{\omega_n} - g_\ell(p_{\omega_n})\| < \frac{1}{2^{(i-\ell+1)+n}}, \quad \text{for}\, 1 \leq n \leq i+1, 1\leq \ell \leq i.
\end{align*}
This completes the proof of the initial claim.

Note that by condition (i), we have
\[
\sum_{n \in \mathbb{N}} \|e^*_{\omega_n}\| \|p_{\omega_n} - e_{\omega_n}\| = \sum_{n \in \mathbb{N}} \|e^*_{\omega_n}\| \|q_{\omega_n}\| \leq \sum_{n \in \mathbb{N}} 2^{-(n+1)} < 1.
\]
Thus, by \cite[Lemma 10.6]{Grosse}, the sequences $(p_{\omega_n})_n$ and $(e_{\omega_n})_n$ are equivalent basic sequences. Now, consider the infinite-dimensional closed subspace $M := \overline{\mathrm{span}}\{p_{\omega_n} : n \in \mathbb{N}\}$ of $X$. We claim that for each $\ell \in \mathbb{N}$:
\[
T^{\psi_{j,\ell}}x\xrightarrow[j\rightarrow \infty]{} g_{\ell}(x), \forall x\in M.
\]
Fix $ \ell \in \mathbb{N} $. Now, consider $ x \in M $. We write $ x = \sum_{n} \beta_n p_{\omega_n} $, where $ (\beta_n)_n \in C_0(\mathbb{N}) $. Recall that $ p_{\omega_n} = e_{\omega_n} + q_{\omega_n} $. It is worth noting that the series $ \sum_n \beta_n e_{\omega_n} $ converges because $ (e_{\omega_n}) $ is a basic sequence equivalent to $ (p_{\omega_n}) $. Moreover, one can observe that, by condition (i), the series $ \sum_n \beta_n q_{\omega_n} $ is absolutely convergent. Thus,
\[
\begin{aligned}
\|T^{\psi_{j,\ell}}x-g_{\ell}(x)\| &= \left\|T^{\psi_{j,\ell}}\left(\sum_{n\leq j+\ell} \beta_{n}p_{\omega_{n}}+\sum_{n>j+\ell} \beta_{n}(e_{\omega_{n}}+q_{\omega_{n}})\right) - g_{\ell}\left(\sum_{n\leq j+\ell}+\sum_{n>j+\ell}\beta_{n}p_{\omega_{n}}\right)\right\| \\
{} &\leq \Vert{\beta\Vert}_{\infty}\left(\sum_{n\leq j+\ell} \Vert{T^{\psi_{j,\ell}}p_{\omega_{n}}-g_{\ell}(p_{\omega_{n}})\Vert}+\sum_{n>j+\ell}\Vert{T^{\psi_{j,\ell}}q_{\omega_{n}}\Vert}\right)\\
{} &{} \hspace{1cm} + \Vert{T^{\psi_{j,\ell}}(\sum_{n>j+\ell}\beta_{n}e_{\omega_{n}})\Vert}+ \Vert{g_{\ell}\Vert\cdot \Vert\sum_{n>j+\ell}\beta_{n}p_{\omega_{n}}\Vert}\\
{} &\leq \Vert{\beta\Vert}_{\infty}\sum_{k>j}\frac{1}{2^{k}}+\Vert{T^{\psi_{j,\ell}}|_{E_{\omega_{j+\ell+1}}}\Vert}\cdot \Vert{\sum_{n>j+\ell}\beta_{n}e_{\omega_{n}}\Vert}+ \Vert{g_{\ell}}\Vert\cdot\Vert\sum_{n>j+\ell}\beta_{n}p_{\omega_{n}}\Vert\\
{} &\leq \frac{1}{2^{j}}\Vert{\beta\Vert}_{\infty}+ 	\sup_{n\geq 1} \max_{\substack{j,\ell \geq 1 \\ j + \ell = n+1}}\|T^{\theta_{j,\ell}}|_{E_{n}}\|\cdot \Vert{\sum_{n>j+\ell}\beta_{n}e_{\omega_{n}}\Vert}+ \Vert{g_{\ell}}\Vert\cdot\Vert\sum_{n>j+\ell}\beta_{n}p_{\omega_{n}}\Vert
\end{aligned}
\]
It is clear that $\|\sum_{n > j + \ell} \beta_n e_{\omega_n}\|$ and $\|\sum_{n > j + \ell} \beta_n p_{\omega_n}\|$ converge to $0$ as $j \to \infty$, since $(e_{\omega_n})$ and $(p_{\omega_n})$ are basic sequences. Therefore, the last term in the previous expression converges to $0$ as $j \to \infty$.
\end{proof}

In order to obtain an analogue of Theorem \ref{gen.sub} for a countable family of operators in $\mathcal{AP}\Omega(T)$, we establish the following result.  

\begin{proposition}\label{commom-AP}
	Let $X$ be a separable infinite-dimensional $F$-space and let $T \in \mathcal{L}(X)$ with $\mathcal{AP}\Omega(T)\neq \emptyset$. For any countable collection $\{g_\ell\}_{\ell \in \mathbb{N}} \subset \mathcal{AP}\Omega(T)$, there exist sets $\mathrm{B}_\ell \in \mathcal{AP}$, one for each $\ell$, such that
	\[
	\bigcap_{\ell \in \mathbb{N}}\Big\{x \in X : \lim_{\substack{n \to \infty \\ n \in \mathrm{B}_\ell}} T^{n}x = g_\ell(x)\Big\}
	\]
	is a dense linear subspace of $X$.
\end{proposition}

To establish the proof of this statement, we first recall Mycielski’s theorem.

\begin{theorem}[Mycielski Theorem \cite{mycielski1964independent, Tan}] \label{Myci}
	Suppose that \( X \) is a separable complete metric space without isolated points, and that for every \( n \in \mathbb{N} \), the set \( \mathcal{R}_{n} \) is residual in the product space \( X^{n} \). Then there is a Mycielski set \( \mathcal{K} \) in \( X \) such that
	\[
	(x_{1}, x_{2}, \ldots, x_{n}) \in \mathcal{R}_{n}
	\]
	for each \( n \in \mathbb{N} \) and any pairwise distinct \( n \) points \( x_{1}, x_{2}, \ldots, x_{n} \) in \( \mathcal{K} \).
\end{theorem}

A set \( \mathcal{K} \) is referred to as a Mycielski set if the intersection of \( \mathcal{K} \) and any nonempty open set \( U \) contains a Cantor set.

\begin{proof}[Proof of Proposition \ref{commom-AP}]
	For each $m \in \mathbb{N}$, Theorem~\ref{T-APh} ensures that
	\[
	\mathcal{R}_{m}:=\bigcap_{\ell \in \mathbb{N}}\mathcal{AP}\text{R}\!\left(\bigoplus_{j=1}^{m}T,\ \bigoplus_{j=1}^{m}g_{\ell}\right)
	\]
	is residual in $X$. By Mycielski’s Theorem, there exists a Mycielski set $\mathcal{M}\subset X$ with $\mathcal{M}^{m}\subset \mathcal{R}_{m}$. Since $X$ is separable, we may choose a countable dense set $\{y_{i}\}_{i\in\mathbb{N}}\subset \mathcal{M}$.
	
	Proceeding inductively, for each $\ell\in\mathbb{N}$ we construct two sequences of positive integers $(a_{n,\ell})_{n}$ and $(r_{n,\ell})_{n}$ such that:
	\begin{itemize}
		\item[(i)] $a_{k+1,\ell}>a_{k,\ell}+kr_{k,\ell}$ for every $k,\ell\in\mathbb{N}$,
		\item[(ii)] $d\!\left(T^{a_{k,\ell}+jr_{k,\ell}}y_{i},\, g_{\ell}y_{i}\right)<1/k$ for $1\leq i,\ell\leq k$ and $0\leq j\leq k$.
	\end{itemize}
	
	Fix $\ell\in\mathbb{N}$ and define
	\[
	\mathrm{B}_{\ell}:=\{\,a_{k,\ell}+jr_{k,\ell}:\ k\in\mathbb{N},\ 0\leq j\leq k\,\}\in\mathcal{AP}.
	\]
	By construction,
	\[
	\{y_{i}: i\in\mathbb{N}\}\subset \bigcap_{\ell\in\mathbb{N}}\Big\{x\in X:\ \lim_{\substack{n\to\infty\\ n\in \mathrm{B}_{\ell}}}T^{n}x=g_{\ell}(x)\Big\}.
	\]
	Since $\{y_{i}\}$ is dense in $X$, the intersection on the right is a dense linear subspace of $X$.
\end{proof}

The following result provides sufficient conditions for common spaceability when dealing with a countable collection of operators in $\mathcal{AP}\Omega(T)$. Proposition \ref{commom-AP} allows us to establish an analogue of the first condition in Theorem \ref{gen.sub}. The proof proceeds along the same lines as that of Theorem \ref{gen.sub}, and will therefore be omitted.

\begin{theorem}\label{condi-AP}
	Let $ X $ be a (real or complex) separable infinite-dimensional Banach space, and let $ T \in \mathcal{L}(X) $. Suppose that $ \Omega(T) $ is non-empty, and consider $ \{g_\ell\}_{\ell \in \mathbb{N}} \subset \mathcal{AP}\Omega(T) $. If for each $ \ell \in \mathbb{N} $, there exist \(\text{B}_{\ell}:= \{a_{k,\ell}+jr_{k,\ell}: k\in \mathbb{N}, 0\leq j\leq k\}\in \mathcal{AP}\) such that:
	\begin{itemize}
		\item The vector subspace \(\displaystyle{	Z := \bigcap_{\ell\in \mathbb{N}} \{x \in X : \lim_{\substack{n \to \infty \\ n \in \text{B}_{\ell}}}T^{n} x =g_\ell(x)\}}\) is dense in $ X $.
		
		\item  There exists a non-increasing sequence \((E_n)_{n \in \mathbb{N}}\) of infinite-dimensional closed subspaces of \(X\) such that,
		\[
		\sup_{n\geq 1} \max_{\substack{k,\ell \geq 1 \\ 0\leq j\leq k\\ k + \ell = n+1}} \|T^{a_{k,\ell}+jr_{k,\ell}}|_{E_{n}}\| < \infty.
		\]
	\end{itemize}
	Then there exists an infinite-dimensional closed subspace $ M $ and for each $ \ell \in \mathbb{N} $, there exists a subsequence $\text{D}_{\ell}\in \mathcal{AP} $ of $\text{B}_{\ell} $ such that
	\[
	\lim_{\substack{n \to \infty \\ n \in \text{D}_{\ell}}}T^{n}x = g_\ell(x), \quad \forall x \in M, \; \forall \ell \in \mathbb{N}.
	\]
\end{theorem}

In Theorems \ref{subhyper-equi} and \ref{subrec-equi}, characterizations of the existence of hypercyclic and recurrent subspaces are given in terms of the essential spectrum. In the result below, Theorem \ref{subspaces}, we remain in the separable setting, as this assumption is required for the use of Theorem \ref{seq}, which plays a key role in establishing common spaceability. Another point to note is that the characterization is formulated via the left essential spectrum rather than the essential spectrum. Nevertheless, for hypercyclic and recurrent operators these two spectra coincide; see \cite{Gonzales,Lopez}.

Theorem \ref{subspaces} is stated for the case where $X$ is a separable infinite-dimensional complex Banach space. The corresponding result for real Banach spaces will be given later in Theorem \ref{subspaces-real}. To establish the complex case, we first present the following auxiliary result, which will be used in its proof.

\begin{lemma}[\cite{Gonzales, Lopez}]\label{diverg}
	Let $X$ be a separable infinite-dimensional complex Banach space, and let $T \in \mathcal{L}(X)$. Suppose that
	\[
	\sigma_{\ell e}(T) \cap \overline{\mathbb{D}} = \emptyset.
	\]
	Then every infinite-dimensional closed subspace $Z \subset X$ contains a vector $x \in Z$ such that
	\[
	\lim_{n \to \infty} \|T^n x\| = \infty.
	\]
\end{lemma}

\begin{theorem}\label{subspaces}
	Let $X$ be a separable infinite-dimensional complex Banach space, and let $T \in \mathcal{L}(X)$. Suppose that $\Omega(T)$ is non-empty. The following assertions are equivalent:
	\begin{enumerate}
		\item There exists $g \in \mathcal{L}(X)$ such that $\text{R}(T,g)$ is spaceable.
		\item For any countable collection $\{g_\ell\}_{\ell\in\mathbb{N}} \subset \Omega(T)$, there exist strictly increasing sequences of positive integers $(\psi_{n,\ell})_n$ and an infinite-dimensional closed subspace $M \subset X$ such that 
		\[
		T^{\psi_{n,\ell}}x \xrightarrow[n \to \infty]{} g_\ell(x), \quad \forall x \in M, \ \forall \ell\in \mathbb{N}.
		\]	
		\item There exists a strictly increasing sequence of positive integers $(\theta_n)_n$ and an infinite-dimensional closed subspace $E \subset X$ such that $\sup_{n} \|T^{\theta_n}|_E\| < \infty$.
		\item The left essential spectrum of $T$ intersects the closed unit disk.
	\end{enumerate}
\end{theorem}

\begin{proof}
	The implication (2) $\Rightarrow$ (1) is immediate, and (2) $\Rightarrow$ (3) follows from the Banach–Steinhaus Theorem. The implications (1) $\Rightarrow$ (4) and (3) $\Rightarrow$ (4) are consequences of Lemma \ref{diverg}. It remains to show that (4) implies (2).
	
	Assume that there exists $\lambda \in \sigma_{\ell e}(T) \cap \overline{\mathbb{D}}$; equivalently, $T-\lambda$ is not a left-Fredholm operator. By \cite[Proposition D.3.4]{livro}, there exist an infinite-dimensional closed subspace $E$ and a compact operator $K \in \mathcal{L}(X)$ such that $(T-K)|_E = \lambda \mathrm{Id}|_E$. In particular, $\|(T-K)|_E\|\leq 1$.
	
	Let $\{g_\ell\}_{\ell \in \mathbb{N}} \subset \Omega(T)$ be any fixed countable family. By Proposition \ref{seq}, for each $\ell \in \mathbb{N}$ there exists a strictly increasing sequence of positive integers $(\theta_{n,\ell})_n$ such that 
	\[
	Z := \bigcap_{\ell \in \mathbb{N}} \{x \in X : T^{\theta_{n,\ell}}x \xrightarrow[n \to \infty]{} g_\ell(x)\}
	\]
	is a dense subspace of $X$.
	
	For each $\theta_{n,\ell}$ we can write $T^{\theta_{n,\ell}} = (T-K)^{\theta_{n,\ell}} + K_{n,\ell}$, where $K_{n,\ell}$ is compact. Consider the sequence $\{A_m\}_{m \in \mathbb{N}}$ of compact operators on $X$ arranged as
	\[
	\begin{array}{cccccccc}
		\overbrace{K_{1,1}}^{j+\ell=2} & ; & 
		\overbrace{K_{1,2}, K_{2,1}}^{j+\ell=3} & ; & 
		\overbrace{K_{1,3}, K_{2,2}, K_{3,1}}^{j+\ell=4} & ; & \cdots & 
	\end{array}
	\]
	According to \cite[Lemma 8.13]{livro}, there exists a non-increasing sequence $\{F_m\}_{m \in \mathbb{N}}$ of finite-codimensional closed subspaces of $E$ such that $\|A_m|_{F_m}\| \leq 1$. Clearly, each $F_m$ is infinite-dimensional. Define $E_n := F_{n^{2}}$, so that
	\[
	\|K_{j,\ell}|_{E_{j+\ell-1}}\| \leq 1, \quad \forall j,\ell \in \mathbb{N}.
	\]
	
	Fix any $n$. Then, for $j+\ell=n+1$,
	\begin{align*}
		\|T^{\theta_{j,\ell}}|_{E_n}\| 
		&= \|(T-K)^{\theta_{j,\ell}}|_{E_n} + K_{j,\ell}|_{E_n}\| \\
		&\leq \|(T-K)^{\theta_{j,\ell}}|_{E_n}\| + \|K_{j,\ell}|_{E_n}\| \\
		&\leq 2.
	\end{align*}
	Since $n$ is arbitrary, it follows that
	\[
	\sup_{n\geq 1} \max_{\substack{j,\ell \geq 1 \\ j + \ell = n+1}} \|T^{\theta_{j,\ell}}|_{E_{n}}\| \leq 2.
	\]
	
	Applying Theorem \ref{gen.sub}, we deduce that for each $\ell \in \mathbb{N}$ there exists a subsequence $(\psi_{n,\ell})_n$ of $(\theta_{n,\ell})_n$ and an infinite-dimensional closed subspace $M \subset X$ such that
	\[
	T^{\psi_{n,\ell}}x \xrightarrow[n \to \infty]{} g_\ell(x), \quad \forall x \in M, \ \forall \ell \in \mathbb{N}.
	\]
	This completes the proof.
\end{proof}

\begin{remark}\label{seq-AP}
	An analogous characterization holds when $\Omega(T)$ is replaced by $\mathcal{AP}\Omega(T)$, in light of Theorem \ref{condi-AP}. 
	In this setting, condition (2) of Theorem \ref{subspaces} must be reformulated as follows: 
	for each $g_\ell \in \mathcal{AP}\Omega(T)$, there exists $\mathrm{B}_\ell \in \mathcal{AP}$ such that
	\[
	\lim_{\substack{n \to \infty \\ n \in \mathrm{B}_\ell}} T^n x = g_\ell(x), \quad \forall x \in M.
	\]
	All the other conditions remain unchanged. 
\end{remark}

\begin{corollary}
	Let $X$ be a separable infinite-dimensional complex Banach space, and let $T \in \mathcal{L}(X)$ with $\Omega(T)\neq\emptyset$. Suppose there exist an infinite-dimensional closed subspace $E\subset X$ and a strictly increasing sequence of positive integers $(\psi_n)$ such that 
	\begin{align*}
		\sup_{n}\|T^{\psi_n}|_{E}\|<\infty.
	\end{align*}
	Then, for every countable collection $\{g_i\}_{i\in\mathbb{N}}\subset \Omega(T)$ simultaneously approximated by $T$, there exist a strictly increasing sequence $(\theta_n)$ and infinite-dimensional closed subspaces $\{M_i\}_{i\in\mathbb{N}}$ such that
	\[
	\lim_{n\to\infty} T^{\theta_n}x = g_i(x), \quad \forall x\in M_i,\ \forall i\in\mathbb{N}.
	\]
\end{corollary}

\begin{proof}
	Let $\{g_i\}_{i \in \mathbb{N}} \subset \Omega(T)$ be simultaneously approximated by $T$. By Theorem \ref{simult}, there exists a strictly increasing sequence of positive integers $(\omega_n)_{n \in \mathbb{N}}$ such that, for each $i \in \mathbb{N}$, the set
	\[
	\{x \in X : T^{\omega_n}x \xrightarrow[n \to \infty]{} g_i(x)\}
	\]
	is dense in $X$.
	
	By Theorem \ref{subspaces}, there exists a subsequence $(\omega_{n,1})_n \subset (\omega_n)_n$ and an infinite-dimensional closed subspace $M_1 \subset X$ such that
	\[
	T^{\omega_{n,1}}x \xrightarrow[n\rightarrow \infty]{} g_1(x), \quad \forall x \in M_1.
	\]
	Proceeding inductively, we obtain subsequences $(\omega_{n,\ell})_n$ of $(\omega_n)_n$ and infinite-dimensional closed subspaces $\{M_\ell\}_{\ell \in \mathbb{N}}$ such that:
	\begin{itemize}
		\item $(\omega_{n,\ell+1})_n$ is a subsequence of $(\omega_{n,\ell})_n$ for each $\ell \in \mathbb{N}$,
		\item $T^{\omega_{n,\ell}}x \to g_\ell(x)$ as $n\to\infty$, for all $x \in M_\ell$ and all $\ell \in \mathbb{N}$.
	\end{itemize}
	
	Finally, for each $n \in \mathbb{N}$, set $\theta_n := \omega_{n,n}$. Clearly, $(\theta_n)_n$ is a subsequence of $(\omega_n)_n$ such that, for every $\ell \in \mathbb{N}$,
	\[
	\lim_{n \to \infty} T^{\theta_n}x = g_\ell(x), \quad \forall x \in M_\ell.
	\]
	This completes the proof.
\end{proof}

We now turn to the case of real Banach spaces. The approach follows A.~López’s ideas on recurrent subspaces for real Banach spaces, making use of the so-called complexification. We next examine this procedure in more detail. Further background and related aspects can be found in \cite{moslehian2022similarities, munoz1999complexifications}.

Let $(X, \|\cdot\|)$ be a real Banach space. Define $\widetilde{X} := \{x + iy : x, y \in X\}$ as the complexification of $X$, which is a vector space with multiplication by complex scalars defined as follows:
\[
(a + ib)(x + iy) := (ax - by) + i(ay + bx),
\]
for any $a, b \in \mathbb{R}$ and $x, y \in X$. Furthermore, if we equip $\widetilde{X}$ with the norm given by
\[
\|x + iy\|_{\mathbb{C}} := \sup_{t \in [0, 2\pi]} \|\cos(t)x - \sin(t)y\|,
\]
then $(\widetilde{X}, \|\cdot\|_{\mathbb{C}})$ becomes a complex Banach space.

In the same vein, if we consider a real linear operator $T : X \to X$ on the real Banach space $X$, there exists a unique complex-linear extension $\widetilde{T} : \widetilde{X} \to \widetilde{X}$ given by
\[
\widetilde{T}(x + iy) := Tx + iTy,
\]
with $\|\widetilde{T}\| = \|T\|$.

Let us first examine some preliminary results to address the case when $ X $ is a real separable Banach space.

\begin{proposition}[\cite{Lopez}]\label{projection}
	Let $ X $ be a real infinite-dimensional Banach space, and let $ N $ be an infinite-dimensional closed subspace of $ \widetilde{X} $. Then the set
	\[
	\{x \in X : \exists\, y \in X \,\text{with}\, x + iy \in N\}
	\]
	contains an infinite-dimensional closed subspace of $ X $.
\end{proposition}

\begin{proposition}\label{complexi}
	Let $ X $ be a real separable infinite-dimensional Banach space, and let $ T \in \mathcal{L}(X) $. Suppose that $ \Omega(T) $ is non-empty. If $ \mathcal{N} $ is a c.s.a by $ T $, then $ \{\widetilde{g} : g \in \mathcal{N}\} $ is a c.s.a by $ \widetilde{T} $.
\end{proposition}

\begin{proof}
	Suppose $ \mathcal{N} $ is a c.s.a by $ T $. Let $ \{\widetilde{g}_\ell\}_{\ell=1}^m $ be any finite collection where $ \{g_\ell\}_{\ell=1}^m \subset \mathcal{N} $. By Proposition \ref{simult}, there exists a strictly increasing sequence of positive integers $ (\theta_n)_n $ such that for each $ \ell \in \{1, \ldots, m\} $, the subspace $ D_\ell := \{x \in X : \lim_{n} T^{\theta_n}x = g_\ell(x)\} $ is dense in $ X $. Similarly, it is straightforward to verify that 
	\[
	D_\ell + iD_\ell \subset \{z \in \widetilde{X} : \lim_{n} \widetilde{T}^{\theta_n}z = \widetilde{g}_\ell(z)\}.
	\]
	This ensures that $ \{\widetilde{g}_\ell\}_{\ell=1}^m $ satisfies the conditions of Proposition \ref{simult}. Since the finite collection $ \{\widetilde{g}_\ell\}_{\ell=1}^m $ was arbitrary, it follows that $ \{\widetilde{g} : g \in \mathcal{N}\} $ is a c.s.a by $ \widetilde{T} $.
\end{proof}

\begin{proposition}[\cite{Lopez}]\label{essential-real}
	Let \(X\) be a real (and not necessarily separable) Banach space and let \(T\in \mathcal{L}(X)\).
	If \(\sigma_{e}(\widetilde{T})\cap \overline{\mathbb{D}}\neq \emptyset\), then every infinite-dimensional closed subspace \(E\subset X\) admits a vector \(x\in E\)
	such that \(\lim_{n}\Vert{T^{n}x}\Vert=\infty\).
\end{proposition}

\begin{theorem}\label{subspaces-real}
Let $ X $ be a real separable infinite-dimensional Banach space, and let $ T $ be a recurrent operator acting on $ X $. Suppose that $ \Omega(T) $ is non-empty. The following assertions are equivalent:
	\begin{enumerate}
		\item There exists $ g \in \mathcal{L}(X)$ such that $ \text{R}(T, g) $ is spaceable.
		
		\item For any denumerable collection $ \{g_\ell\}_{\ell\in \mathbb{N}} \subset \Omega(T) $, there exist strictly increasing sequences of positive integers $ (\theta_{n,\ell})_n $ and an infinite-dimensional closed subspace $ M \subset X $ such that
		\begin{align*}
			T^{\theta_{n,\ell}}x \xrightarrow[n \to \infty]{} g_\ell(x), \quad \forall x \in M, \forall \ell\in \mathbb{N}.
		\end{align*}	
		\item There exists a strictly increasing sequence of positive integers $ (\theta_n)_n $ and an infinite-dimensional closed subspace $ E \subset X $ such that \(\sup_{n} \|T^{\theta_n}|_E\| < \infty\).
		
		\item The essential spectrum of \(\widetilde{T}\) intersects the closed unit disk.
	\end{enumerate}
\end{theorem}

\begin{proof}
	The implication from (2) to (1) is immediate. The implication from (2) to (4) follows from the Banach--Steinhaus Theorem, while Proposition \ref{essential-real} ensures that (1) implies (4) and, similarly, that (3) implies (4). It remains to show that (4) implies (2), for which we will rely on Theorem \ref{subspaces}.
	
	Assume that (4) holds. Since $T$ is recurrent, it follows that $\widetilde{T}$ is also recurrent, and hence $\sigma_{\ell e}(\widetilde{T}) = \sigma_e(\widetilde{T})$. Now, fix any countable collection $\{g_\ell\}_{\ell \in \mathbb{N}} \subset \Omega(T)$. By Proposition \ref{complexi}, we have $\{\widetilde{g}_\ell\}_{\ell \in \mathbb{N}} \subset \Omega(\widetilde{T})$. Therefore, by Theorem \ref{subspaces}, for each $\ell \in \mathbb{N}$ there exists a strictly increasing sequence of positive integers $(\theta_{n,\ell})_{n \in \mathbb{N}}$ and an infinite-dimensional closed subspace $N \subset \widetilde{X}$ such that
	\[
	\lim_{n \to \infty} \widetilde{T}^{\theta_{n,\ell}} z = \widetilde{g}_\ell(z), \quad \forall z \in N, \ \forall \ell \in \mathbb{N}.
	\]
	Finally, by Proposition \ref{projection}, there exists an infinite-dimensional closed subspace $M \subset X$ such that
	\[
	\lim_{n \to \infty} T^{\theta_{n,\ell}} x = g_\ell(x), \quad \forall x \in M, \ \forall \ell \in \mathbb{N}.
	\]
	This completes the proof.
\end{proof}

\begin{corollary}
	Let $X$ be a real separable infinite-dimensional Banach space, and let $T$ be a recurrent operator acting on $X$. Suppose that $\Omega(T)$ is non-empty. Assume that there exist an infinite-dimensional closed subspace $E \subset X$ and a strictly increasing sequence of positive integers $(\psi_n)$ such that 
	\begin{align*}
		\sup_{n}\|T^{\psi_n}|_{E}\| < \infty.
	\end{align*}
	Then, for every countable collection $\{g_i\}_{i \in \mathbb{N}} \subset \Omega(T)$ simultaneously approximated by $T$, there exist a strictly increasing sequence $(\theta_n)$ and infinite-dimensional closed subspaces $\{M_i\}_{i \in \mathbb{N}}$ such that
	\[
	\lim_{n \to \infty} T^{\theta_n} x = g_i(x), \quad \forall x \in M_i,\ \forall i \in \mathbb{N}.
	\]
\end{corollary}

\begin{proof}
	Fix any countable collection $\{g_\ell\}_{\ell \in \mathbb{N}} \subset \Omega(T)$ simultaneously approximated by $T$. By Proposition \ref{complexi}, the family $\{\widetilde{g}_\ell\}_{\ell \in \mathbb{N}} \subset \Omega(\widetilde{T})$ is also a countable collection simultaneously approximated by $\widetilde{T}$. By Theorem \ref{subspaces}, there exist a strictly increasing sequence $(\theta_n)$ and infinite-dimensional closed subspaces $\{N_\ell\}_{\ell \in \mathbb{N}}$ of $\widetilde{X}$ such that, for each $\ell \in \mathbb{N}$,
	\[
	\lim_{n \to \infty} \widetilde{T}^{\theta_n} z = \widetilde{g}_\ell(z), \quad \forall z \in N_\ell.
	\]
	Finally, Proposition \ref{projection} applied to each $N_\ell$ ensures the existence of infinite-dimensional closed subspaces $M_\ell \subset X$ such that, for every $\ell \in \mathbb{N}$,
	\[
	\lim_{n \to \infty} T^{\theta_n} x = g_\ell(x), \quad \forall x \in M_\ell.
	\]
	This completes the proof.
\end{proof}

To conclude this section, we present some consequences obtained under the additional assumption of $\mathrm{SOT}$-separability of certain subsets of $\Omega(T)$.

\begin{lemma}\label{separable}
	Let $X$ be a separable infinite-dimensional $F$-space, and let $T \in \mathcal{L}(X)$ with $\Omega(T)\neq\emptyset$. Suppose $F\subset \Omega(T)$ is $\text{SOT}$-separable. If $\{g_{\ell}\}_{\ell\in\mathbb{N}}\subset F$ is such that 
	\[
	F=\overline{\{g_{\ell}\}_{\ell\in\mathbb{N}}}^{\,\text{SOT}},
	\]
	then
	\[
	\bigcap_{h\in F}\text{R}(T,h)=\bigcap_{\ell\in\mathbb{N}}\text{R}(T,g_{\ell}).
	\]
\end{lemma}

\begin{proof}
	Fix an arbitrary $x \in \bigcap_{\ell\in\mathbb{N}}\text{R}(T,g_{\ell})$. We will show that $x \in \text{R}(T,h)$ for every $h \in F$. Let $h \in F$ be arbitrary. Given $\varepsilon > 0$, consider the $\text{SOT}$-basic neighborhood
	\[
	\mathcal{N}(h,x,\varepsilon):=\{g \in \mathcal{L}(X): d(g(x), h(x))<\varepsilon\}.
	\]
	By hypothesis, there exists some $\ell \in \mathbb{N}$ such that $g_{\ell}\in \mathcal{N}(h,x,\varepsilon)$. Since $x\in \text{R}(T,g_{\ell})$, there exists a strictly increasing sequence $(\theta_{n})_{n}$ such that
	\[
	T^{\theta_{n}}x \longrightarrow g_{\ell}(x) \quad \text{as } n\to\infty.
	\]
	In particular, for sufficiently large $n$, $T^{\theta_{n}}x \in B(h(x),\varepsilon)$. Hence $x \in \mathrm{R}(T,h)$, as desired.
\end{proof}

\begin{theorem}\label{O-c.d.l}
	Let $ X $ be a  separable infinite-dimensional $ F $-space, and let $ T \in \mathcal{L}(X) $. If $ \Omega(T) $ is non-empty and $ \mathrm{SOT}$-separable in $ \mathcal{L}(X) $, then the following intersection
	\[
	\bigcap_{h \in \Omega(T)} \text{R}(T, h)
	\]
	is dense-lineable.
\end{theorem}

\begin{proof}
	Suppose that $ \Omega(T) $ is non-empty and $ \mathrm{SOT} $-separable. Then there exists a countable  \(\{g_\ell\}_{\ell \in \mathbb{N}}\) contenida en \(\Omega(T) \) such that \(\Omega(T) = \overline{\{g_\ell\}_{\ell \in \mathbb{N}}}^{\mathrm{SOT}}\). By Proposition \ref{seq}, for each $ \ell \in \mathbb{N} $, there exists a strictly increasing sequence of positive integers $ (\theta_{n,\ell})_n $ such that the set
	\[
	D := \bigcap_{\ell \in \mathbb{N}} \{x \in X : T^{\theta_{n,\ell}}x \xrightarrow[n \to \infty]{} g_\ell(x)\}
	\]
	is a dense subspace of $ X $. Therefore, by Lemma \ref{separable}
	\[
	D \subset \bigcap_{\ell \in \mathbb{N}} \text{R}(T, g_\ell) = \bigcap_{h \in \Omega(T)} \text{R}(T, h).
	\]
This concludes the proof since $ D $ is a dense subspace of $ X $.
\end{proof}

\begin{theorem}\label{main-Thm}
	Let $X$ be a complex separable infinite-dimensional Banach space, and let $T \in \mathcal{L}(X)$. Suppose that $\Omega(T)$ is non-empty. If there exists a strictly increasing sequence of positive integers $(\theta_n)_n$ and an infinite-dimensional closed subspace $E \subset X$ such that
	\[
	\sup_n \|T^{\theta_n}|_E\| < \infty,
	\]
	then for any SOT-separable subset $F \subset \Omega(T)$,
	\begin{align}\label{main-inter}
			\bigcap_{h \in F} \text{R}(T, h)
	\end{align}
	is spaceable.
\end{theorem}

\begin{proof}
	Let $F \subset \Omega(T)$ be a $\mathrm{SOT}$-separable set. Then there exists a countable subset $\{g_\ell\}_{\ell \in \mathbb{N}} \subset F$ such that \(	F = \overline{\{g_\ell\}_{\ell \in \mathbb{N}}}^{\mathrm{SOT}}\). By Theorem \ref{subspaces}, for each $\ell \in \mathbb{N}$ there exists a strictly increasing sequence $(\theta_{n,\ell})_{n}$ and an infinite-dimensional closed subspace $M \subset X$ such that \(\lim_{n} T^{\theta_{n,\ell}}x = g_\ell(x), \forall x \in M, \ \forall \ell \in \mathbb{N}\). Hence,
	\[
	M \subset \bigcap_{\ell \in \mathbb{N}} \text{R}(T,g_\ell).
	\]
	By Lemma \ref{separable}, it follows that
	\[
	M \subset \bigcap_{h \in F} \text{R}(T,h).
	\]
	This completes the proof.
\end{proof}

\begin{remark}\label{remark-main}
	The conclusion of Theorem \ref{main-Thm} remains valid when $\Omega(T)$ is replaced by $\mathcal{AP}\Omega(T)$, as observed in Remark \ref{seq-AP}; in this case, the sets in \eqref{main-inter} are of the form $\mathcal{AP}\text{R}(T,h)$. Moreover, if $T$ is a recurrent operator on a real separable infinite-dimensional Banach space with $\Omega(T) \neq \emptyset$, the statement of Theorem \ref{main-Thm} also holds by virtue of Theorem \ref{subspaces-real}. The same applies when considering $\mathcal{AP}\Omega(T)$.
\end{remark}

Let \(X\) be a (real or complex) separable infinite-dimensional Banach space and let \(T \in \mathcal{L}(X)\). We focus on two representative cases: when \(T\) is quasi-rigid and when \(T\) is weakly mixing. The existence of recurrent and hypercyclic subspaces is addressed by applying Theorem \ref{main-Thm} and Remark \ref{remark-main} for \(F = \{\mathrm{Id}\}\) and \(F = \mathcal{L}(X)\), respectively. 

If $T$ is such that $T \oplus T$ is $\mathcal{AP}$-hypercyclic, then by the equivalences established in \cite{cardeccia2022multiple} it follows that $\mathcal{AP}\Omega(T) = \mathcal{L}(X)$. This immediately yields the following consequence.

\begin{corollary}
	Let $X$ be a (real or complex) separable infinite-dimensional Banach space, and let $T \in \mathcal{L}(X)$. If $T \oplus T$ is $\mathcal{AP}$-hypercyclic, then:
	\begin{enumerate}
		\item $T$ admits an $\mathcal{AP}$-recurrent subspace.
		\item $T$ admits an $\mathcal{AP}$-hypercyclic subspace.
	\end{enumerate}
\end{corollary}

\begin{example}
	Let $T$ be a hypercyclic operator on a separable infinite-dimensional Banach space. 
	If $T$ admits a recurrent subspace, then
	\begin{align*}
		\bigcap_{\lambda\in \mathbb{K}}\{x\in X : \exists\, \theta_{n}\uparrow \infty \text{ such that } T^{\theta_{n}}x \xrightarrow[n\to\infty]{} \lambda x\}
	\end{align*}
	is spaceable. In other words, the set above can be viewed as 
	$\bigcap_{\lambda \in \mathbb{K}} \text{R}(T, \lambda \text{Id})$, 
	which is spaceable by Theorem~\ref{main-Thm}.
\end{example}

The following result is motivated by the open problem posed by C. Gilmore in \cite[Question 7]{gilmore2020linear}, which concerns characterizing when a hypercyclic operator that is not weakly mixing admits a hypercyclic subspace. 

\begin{theorem}\label{ava-partial}
	Let $ X $ be a (real or complex) separable infinite-dimensional Banach space, and let $ T \in \mathcal{L}(X) $ be a hypercyclic operator. Suppose that $ \Omega(T) $ is $\mathrm{SOT}$-separable and \(T\) admit a recurrent subspace. Then there exists an infinite-dimensional closed subspace $ E \subset X $ with a Schauder basis $ (e_n)_{n \in \mathbb{N}} $ such that:
	\[
	\{e_n : n \in \mathbb{N}\} \subset \text{HC}(T),
	\]
	and
	\[
	E \subset \bigcap_{h \in \Omega(T)} \text{R}(T, h).
	\]
	Moreover, for each $ x \in E $, the set $ \overline{\{T^n x\}_{n}} $ is a cone containing the vector subspace $\overline{\text{span}(\{T^n x\}_{n})}$.
\end{theorem}

We conclude this section with some remarks and open problems.

\noindent
It is worth noting that the spaceability results established in this paper have been obtained in the framework of Banach spaces. A natural question is whether analogous statements can be extended to the setting of Fréchet spaces. In particular, A. López \cite[Question 7.7]{Lopez} explicitly asks whether Theorem \ref{condi-lo} remains valid for Fréchet spaces, with the appropriate modifications. Developing a suitable technique to address this problem would naturally lead to a plausible extension of Theorem \ref{gen.sub} to the broader context of Fréchet spaces. In this direction, we also mention the works \cite{menet2013hypercyclic, menet2014hypercyclic, petersson2006hypercyclic} on hypercyclic subspaces in Fréchet spaces.

In addition, the results obtained here concerning the Furstenberg family $\mathcal{AP}$ suggest further developments. The contributions in \cite{bes2015existence, bonilla2012frequently, menet2015existence} on frequently hypercyclic subspaces and upper frequently hypercyclic subspaces point toward a natural line of research, namely, to establish sufficient conditions for the existence of upper frequently recurrent subspaces. Finally, let us recall that the spaceability within $\Omega(T)$ relies on the common dense-lineability established in \cite{arbieto2025dense}. Moreover, the authors provide sufficient conditions to obtain common dense-lineability in the context of Furstenberg families.

\subsection*{Acknowledgment}
The first author was partially supported by CNPq- Edital Universal Proc. 407854/2021-5. The second author acknowledges financial support from CNPq-Brazil and the Fundação Carlos Chagas Filho de Amparo à Pesquisa do Estado do Rio de Janeiro (FAPERJ) through grants E-26/210.388/2019 and E-26/204.324/2024.

\bibliographystyle{abbrv}
\bibliography{spaceability}
	
\end{document}